\documentclass[11pt,letterpaper]{amsart}

\usepackage{amsmath}
\usepackage{enumerate}
\usepackage{amsthm}
\usepackage{url}
\usepackage[margin=1.0in]{geometry}
\usepackage{textcomp}

\newcommand{\cbundle}[1]{\overline{#1}}

\usepackage{amssymb}
\usepackage{amsfonts}

\newcommand{\tQ}{\widetilde{Q}}
\newcommand{\tM}{\widetilde{M}}
\newcommand{\tX}{\widetilde{X}}
\newcommand{\tS}{\widetilde{S}}
\newcommand{\tU}{\widetilde{U}}

\newcommand{\tA}{\widetilde{A}}

\newcommand{\ts}{\tilde{s}}

\newcommand{\brh}{\bar{h}}

\newcommand{\ctQ}[1]{\underline{\tQ_{#1}}}

\newcommand{\bM}{\breve{M}}
\newcommand{\bX}{\breve{X}}
\newcommand{\bU}{\breve{U}}

\newcommand{\bh}{\breve{h}}
\newcommand{\bg}{\breve{g}}

\newtheorem{theorem}{Theorem}[section]
\newtheorem{lemma}[theorem]{Lemma}

\newtheorem{proposition}[theorem]{Proposition}
\newtheorem{corollary}[theorem]{Corollary}
\newtheorem{definition}[theorem]{Definition}

\newenvironment{remark}[1][Remark.]{\begin{trivlist}
\item[\hskip \labelsep {\bfseries #1}]}{\end{trivlist}}

\DeclareMathOperator{\grad}{grad}

\DeclareMathOperator{\id}{id}
\DeclareMathOperator{\spn}{span}

\newcommand\numberthis{\addtocounter{equation}{1}\tag{\theequation}}

\title[Cornered Asymptotically Hyperbolic Metrics]{Exponential Map and Normal Form for Cornered Asymptotically Hyperbolic Metrics}
\author{Stephen E. McKeown}
\address{Department of Mathematics\\University of Washington\\Seattle, Washington\\USA}
\curraddr{Department of Mathematics\\Princeton University\\Fine Hall\\Washington Road\\Princeton, New Jersey 08544, USA}
\email{smckeown@math.princeton.edu}
\thanks{Research partially supported by NSF RTG Grant DMS-0838212 and Grant DMS-1161283}

\begin{document}
\begin{abstract}
  This paper considers asymptotically hyperbolic manifolds with a finite boundary intersecting the usual
  infinite boundary -- \emph{cornered asymptotically hyperbolic manifolds} -- 
and proves a theorem of Cartan-Hadamard type near infinity for the normal exponential map on the 
finite boundary.
As a main application, a normal form for such manifolds at the corner is then constructed,
analogous to the normal form for usual asymptotically hyperbolic manifolds and suited to studying
geometry at the corner. The normal form is at the same time a submanifold normal form near the finite boundary
and an asymptotically hyperbolic normal form near the infinite boundary.
\end{abstract}

\maketitle

\section{Introduction}

A foundational fact of Riemannian geometry is that the exponential map at a point of a manifold
$(X,g)$ is a diffeomorphism on a neighborhood
of $0$; and similarly, the normal exponential map associated to a hypersurface $\iota: Q \hookrightarrow X$
is a diffeomorphism on a neighborhood
of the zero section. The Cartan-Hadamard theorem gives a spectacular global extension of the former of these 
in case $X$ is of nonpositive curvature and complete, to wit that the exponential is a covering map.
The global situation for the normal exponential map in a negatively curved space is more subtle due
to the importance of the geometry and topology of $Q$, and has developed more slowly. Thus \cite{her63} showed,
for example, that if $X$ is complete and of nonpositive curvature, 
$Q$ is closed, connected, and totally geodesic, and $\iota_*\pi_1(Q) = \pi_1(X)$, then $\exp$ is a diffeomorphism.
A version for level sets $Q$ of convex functions on complete nonpositively curved manifolds, among other related theorems,
can be found in the expansive \cite{bo69}, and in \cite{lang99}, it is proved (and stated to be generally known
but unpublished) that if $X$ is complete of nonpositive curvature and $Q$ totally geodesic, the normal exponential map over $Q$ is a diffeomorphism
onto its image. More recently, in \cite{bm08} it is proved that if $X^m$ is complete and $K^m \subset X^m$
is a compact, totally convex submanifold with boundary such that $X \setminus K$ has pinched negative curvature,
then the normal exponential map over $Q = \partial K$ is a diffeomorphism onto $\overline{X \setminus K}$.

Asymptotically hyperbolic (AH) manifolds $(X,M = \partial X,g)$
are complete but may have arbitrary curvature on a compact set, with curvature
approaching $-1$ toward the boundary at infinity. (Henceforth, $X$ will refer to a manifold with
boundary, and the metric of interest will live on the interior $\mathring{X}$).
Much of the interesting geometry and analysis on these spaces
occurs near the boundary, so results, such as that just mentioned in \cite{bm08}, that allow conclusions
about a collar neighborhood of the boundary can play a role analogous in this context to that of global results
such as Cartan-Hadamard in the negatively curved setting. The most important of these is the existence of
the geodesic normal form, first proved in \cite{gl91}: suppose $(M,[h])$ is the conformal infinity of $X$ and 
that $h \in [h]$. 
Then for $\varepsilon > 0$ small, there is a unique diffeomorphism $\psi$ from $[0,\varepsilon)_r \times M$
to a neighborhood of $M$ in $X$ such that $\psi^*g = \frac{dr^2 + g_r}{r^2}$ and $g_0 = h$. The normal form
has been a central tool in studying the duality between the boundary and interior geometry of asymptotically hyperbolic
Einstein manifolds, and has frequently been employed in studying the analysis and geometry of AH manifolds
generally. In this paper we prove a theorem of hypersurface Cartan-Hadamard type near the corner
for asymptotically hyperbolic manifolds that have a finite boundary in addition to the usual infinite boundary,
and then use this to construct a normal form at the corner.

We define a cornered space as a manifold $X$
with two boundary components $M$ and $Q$ that meet in a codimension-two corner
$S = Q \cap M \neq \emptyset$, and a cornered AH (CAH) space as a cornered space equipped on the interior
with a metric $g_+$ such that $g_+$ is smooth and nondegenerate at $Q \setminus S$ but
asymptotically hyperbolic at $M$. Such manifolds arose in the proof (\cite{bh14}) of local regularity for
AH Einstein manifolds, since a small neighborhood of a boundary point on a global asymptotically hyperbolic manifold
has such a structure. Such manifolds have also been studied in the physics literature in the context
of a proposed AdS/CFT-type correspondence for the case when the conformal field theory lives on a space
with boundary (BCFT). See \cite{ntu12} and the references therein. Like a usual AH metric,
CAH metrics have a conformal infinity $[h]$ on $M$.

The paradigm example of such a space is a portion of hyperbolic space bounded by an umbilic hypersurface.
Let $\mathbb{H}^{n + 1} = \left\{ x^0 > 0 \right\}$ be the upper half-space model
of hyperbolic space, with $g_+$ the hyperbolic metric $g_+ = \frac{(dx^0)^2 + \dots + (dx^n)^2}{(x^0)^2}$.
Let $\alpha \in \mathbb{R}$, and $X = \{(x^0,\dots,x^n) \in \overline{\mathbb{H}^{n + 1}}: x^n \geq \alpha x^0\}$, with
$Q = \left\{ x^n = \alpha x^0\right\}$ and $M = \left\{ x^0 = 0 \text{ and } x^n \geq 0 \right\}$. The conformal infinity
$[h]$ is that of the Euclidean metric on $M$. The geodesics
normal to $Q$ are precisely the intersections with $X$ of the circles $(x^0)^2 + (x^n)^2 = a^2$
(where $a \in \mathbb{R}^{>0}$), $x^1 = x^2 = \dots = x^{n - 1} = const$.
The corner normal form in the hyperbolic case is obtained by introducing polar coordinates $(\theta, \rho)$, in which
$Q, M$, and $S$ are all given by constant coordinates: 
\begin{equation*}
  x^0 = \rho\sin\theta, \quad x^n = \rho\cos\theta.
\end{equation*}
In these coordinates, the metric takes the form
\begin{equation}
  \label{hyppolg}
  g_+ = \csc^2(\theta)\left[ d\theta^2 + \frac{d\rho^2 + (dx^1)^2 + \dots + (dx^{n - 1})^2}{\rho^2}\right].
\end{equation}

The appearance of polar coordinates motivates us in the general case to follow
the usual expedient of blowing up $X$ along $S$, obtaining a blown up space $(\tX,\tM,\tQ,\tS)$, with a blow-down
map $b:\tX \to X$. This has the properties that $b|_{\tX \setminus \tS}:\tX \setminus \tS \to X \setminus S$ is a diffeomorphism,
as are $b|_{\tM}:\tM\to M $ and $b|_{\tQ}:\tQ\to Q$,
while $b|_{\tS}:\tS \to S$ is a fibration with fibers diffeomorphic to the closed unit interval. We will denote
such a diffeomorphism equivalence by $\tX \setminus \tS \approx X \setminus S$ (for example).

In applications of our normal form theorems, we will need to consider metrics smooth
on the blowup but not on the base. 
Thus, we give results for a somewhat wider class of metrics than those of the form 
$b^*g_+$ for $g_+$ a smooth cornered AH metric on $X$.
In Definition \ref{admis} we define admissible metrics, which differ from such a pullback by a perturbation that is smooth on
$\tX$ and vanishes in an appropriate sense at $\tM$ and $\tS$. Thus, such
a metric may be written $g = b^*g_+ + \mathcal{L}$ for appropriate $\mathcal{L}$.
Given any admissible metric $g$, there is a well-defined angle function $\Theta$ on $\tS$, which serves
as a fiber coordinate.

The normal exponential map $\exp$ of $Q \setminus S \approx \tQ \setminus \tS$ is defined on the inward-pointing
normal ray bundle $N_+(\tQ \setminus \tS)$. With
$\nu$ the inward-pointing unit normal field on $\tQ \setminus \tS$, this bundle has a natural decomposition
$N_+(\tQ\setminus\tS) \approx [0,\infty)_t \times (\tQ \setminus \tS)$ given by the prescription
$(t,q) \mapsto t\nu_q$. We compactify $N_+(\tQ \setminus \tS)$ by adding faces corresponding
to $t = \infty$ and to $[0,\infty] \times (\tQ \cap \tS)$, and we denote the compactification by 
$\cbundle{N_+(\tQ \setminus \tS)}$, a manifold with corners of codimension two.

Our first main result is as follows.

\begin{theorem}
  \label{mainthm}
  Let $(\tX,\tM,\tQ,\tS)$ be the blowup of a cornered space, and $g$ an admissible metric on $\tX$.
  There is a neighborhood $V$ of $\tQ \cap \tS$ in $\tQ$ and a neighborhood $\tU$ of $\tS$ in $\tX$ such that $\exp$
  extends to a diffeomorphism $\exp:\cbundle{N_+(V\setminus\tS)} \to \tU$.
\end{theorem}

One of the consequences of this theorem is that near $S$ there is a distinguished
representative of the conformal infinity $[h]$ on $M \setminus S$, which itself is conformally compact on $(M,S)$.
To see this, simply note that $e^{-t}$ is a defining function for $\tM$ via the diffeomorphism in the theorem,
so that $e^{-2t}g|_{T\tM}$ is a well-defined element of 
$(b|_{\tM})^*[h]$ on $\tM \approx M$ depending only on the geometry of $(\tX,g)$. We
call this the induced metric on $M$.

Theorem \ref{mainthm} allows us to prove two normal form theorems. The first applies generally, while the second requires
that $Q$ and $M$ make a constant angle with respect to the compactified $g_+$, but gives a normal form with better
properties.

\begin{theorem}
  \label{normform}
  Let $(\tX,\tM,\tQ,\tS)$ be the blowup of a cornered space, and $g$ an admissible metric on $\tX$.
  For sufficiently small neighborhoods
  $V$ of $\tQ \cap \tS$ in $\tQ$, there exist a neighborhood
  $\tU$ of $\tS$ in $\tX$ and a unique diffeomorphism $\psi:[0,1]_{u} \times V
  \to \tU$ such that $\psi|_{ \left\{ 1 \right\} \times V} = \id_V$ and
  \begin{equation}
    \label{normformeq}
    \psi^*g = \frac{du^2 + h_u}{u^2},
  \end{equation}
  with $h_u$ ($0 \leq u \leq 1)$ 
  a smooth one-parameter family of smooth conformally compact metrics on $(V,\tQ \cap \tS)$, and such that
  $\tM = \psi(\left\{ u = 0 \right\})$ and $\tQ = \psi(\left\{ u = 1 \right\})$.
\end{theorem}

It will be useful in some applications to fix $\tM$ instead of $\tQ$:
\begin{corollary}
  \label{normformcor}
  Let $(\tX,\tM,\tQ,\tS)$ and $g$ be as in Theorem \ref{normform}.
  For sufficiently small neighborhoods $W$ of $\tM \cap \tS$ in $\tM$,
  there exist a neighborhood $\tU$ of $\tS$ in $\tX$ and a unique diffeomorphism $\zeta:[0,1]_u \times W \to \tU$ such that
  $\zeta|_{ \left\{ 0 \right\} \times W} = \id_W$ and so that 
  \begin{equation}
    \zeta^*g = \frac{du^2 + h_u}{u^2},
  \end{equation}
  with $h_u$ ($0 \leq u \leq 1$) a smooth one-parameter family of smooth conformally compact metrics on $(W,\tM \cap \tS)$,
  and such that $\tM = \zeta(\left\{ u = 0 \right\})$ and
  $\tQ = \zeta(\left\{ u = 1 \right\})$.
\end{corollary}

Notice that 
(\ref{normformeq}) is in normal form in the usual asymptotically hyperbolic sense relative to $\tM$, 
while under the subsitution $t = -\log u$,
it is in the usual geodesic normal form relative to $\tQ$. In particular, $t$ is the distance to $\tQ$.

The metrics $h_u$ for fixed $u$ are generally not asymptotically hyperbolic. 
The asymptotic curvature depends on both $u$ and the angle between $Q$ and $M$ at the point
of $S$ approached. However, when the two boundary components make constant angle $\theta_0$, 
we can make the change of variable
$u = \frac{\csc\theta - \cot\theta}{\csc\theta_0 - \cot\theta_0}$ to obtain a normal form with AH slice metrics. 

\begin{theorem}
  \label{normformconst}
  Let $(\tX,\tM,\tQ,\tS)$ be the blowup of the cornered space $(X,M,Q)$, and $g = b^*g_+ + \mathcal{L}$ an admissible metric on $\tX$.
  Suppose that there is some $\theta_0 \in (0,\pi)$ such that, for any defining function $\varphi$ for $M$, the boundary components
  $M$ and $Q$ make constant angle $\theta_0$ with respect to the compactified metric $\varphi^2g_+$.

  For sufficiently small neighborhoods $V$ of $\tQ \cap \tS$ in $\tQ$, there is a neighborhood $\tU$ of $\tS$ in
  $\tX$ and a unique diffeomorphism $\psi:[0,\theta_0]_{\theta} \times V \to \tU$ such that $\psi|_{ \left\{ \theta_0 \right\} \times V} = \id_V$ and
    \begin{equation*}
      \psi^*g = \frac{d\theta^2 + h_{\theta}}{\sin^2\theta},
    \end{equation*}
    where $h_{\theta}$ ($0 \leq \theta \leq \theta_0$) 
    is a smooth one-parameter family of smooth asymptotically hyperbolic metrics on
    $(V,\tQ \cap \tM)$, and such that $\tM = \psi(\{\theta = 0\})$ and $\tQ = \psi(\{\theta = \theta_0\})$. Moreover, 
    $\theta|_{[0,\theta_0]\times(\tQ \cap \tS)} = \psi^*\Theta$.
    Also $\partial_{\theta}\brh_{\theta}|_{\rho = 0} = 0$, where $\brh_{\theta} = \rho^2h_{\theta}$ and $\rho$ is 
    any defining function for $\tQ \cap \tS$ in $V$.
\end{theorem}
Note that the normal form here given, and $\theta_0$, depend only on $g$, and not on the decomposition
$g = b^*g_+ + \mathcal{L}$.

Once again, it can be helpful to fix $\tM$ instead of $\tQ$.
\begin{corollary}
  \label{normformconstcor}
  Let $(\tX,\tM,\tQ,\tS)$, $(X,M,Q)$, and $g$ be as in Theorem \ref{normformconst}, with again a constant angle
  $\theta_0$ between $Q$ and $M$.

  For sufficiently small neighborhoods $W$ of $\tM \cap \tS$ in $\tM$, there is a neighborhood $\tU$ of $\tS$ in
  $\tX$ and a unique diffeomorphism $\zeta:[0,\theta_0]_{\theta} \times W \to \tU$ such that $\zeta|_{ \left\{ 0 \right\} \times W} = \id_W$ and
    \begin{equation*}
      \zeta^*g = \frac{d\theta^2 + h_{\theta}}{\sin^2\theta},
    \end{equation*}
    where $h_{\theta}$ ($0 \leq \theta \leq \theta_0$) 
    is a smooth one-parameter family of smooth asymptotically hyperbolic metrics on
    $(W, \tM \cap \tS)$, and such that $\tM = \zeta(\{\theta = 0\})$ and $\tQ = \zeta(\{\theta = \theta_0\})$. Moreover, $\theta|_{[0,\theta_0] \times (\tM \cap \tS)} 
    = \zeta^*\Theta$, and $\partial_{\theta}\brh_{\theta}|_{\rho = 0} = 0$, where $\brh_{\theta} = \rho^2h_{\theta}$ and $\rho$ is 
    any defining function for $\tM \cap \tS$ in $W$.
\end{corollary}

We can put this in a yet more refined form. For the following, we let $[k] = \left\{ h|_{TS}: h \in [h] \right\}$;
so $[k]$ is a conformal class of metrics on $S$.

\begin{corollary}
  \label{polarcor}
  Let $(\tX,\tM,\tQ,\tS)$, $(X,M,Q)$, and $g$ be as in Theorem \ref{normformconst}, with again a constant angle
  $\theta_0$ between $Q$ and $M$.
  For any $k \in [k]$ and for sufficiently small
  $\varepsilon > 0$, there is a neighborhood $\tU$ of $\tS$ in $\tX$
  and a unique diffeomorphism $\chi:[0,\theta_0]_{\theta} \times S \times [0,\varepsilon)_{\rho} \to \tU$ such that
  $b\circ\chi|_{ \left\{ 0 \right\}\times S \times \left\{ 0 \right\}} = \id_S$ and
    \begin{equation*}
      \chi^*g = \frac{d\theta^2 + h_{\theta}}{\sin^2\theta},
    \end{equation*}
    where $h_{\theta}$ is a smooth one-parameter family of smooth AH metrics on $S \times [0,\varepsilon)$ with
    \begin{equation*}
      h_0 = \frac{d\rho^2 + k_{\rho}}{\rho^2},
    \end{equation*}
    where $k_{\rho}$ is a smooth one-parameter family of smooth metrics on $S$ with $k_0 = k$, and where 
    $\tM = \chi(\left\{ \theta = 0 \right\})$, $\tQ = \chi(\left\{ \theta = \theta_0 \right\})$, and
    $\tS = \chi(\left\{ \rho = 0 \right\})$. Moreover, $\partial_{\theta}(\bar{h}_{\theta}|_{\rho = 0}) = 0$,
    where $\brh = \rho^2h$.
\end{corollary}
Notice this generalizes the form of the hyperbolic metric in
(\ref{hyppolg}).

A normal form of a similar kind for edge spaces was constructed in \cite{gk12}. However, the normal form derived
there corresponds to a flow transverse to the edge boundary, whereas the flow generated by $u$ in (\ref{normformeq})
is tangent to the edge face $\tS$. Another difference is that the normal form constructed here takes a special form
at two different faces, $\tM$ and $\tQ$, as opposed to one.

The paper is organized as follows.  
In Section \ref{setup} we define cornered asymptotically hyperbolic manifolds and their blowups, as well
as construct a class of product decompositions that will be ubiquitous throughout the paper. We note that the
blown-up face $\tS$ has an edge structure in the sense of \cite{mel08}, which meets the AH face $\tM$, and we define 0-edge bundles
as the natural bundles associated to this structure. We then use this to
define and discuss admissible metrics. In Section \ref{geods}, inspired by the convexity arguments of \cite{bo69},
we use a natural asymptotic solution to $\nabla^2_gw = wg$ to derive the central properties of the $g$-geodesics 
leaving
$\tQ$ normally. Our result shows that they approximately generalize the behavior of the analogous 
geodesics in hyperbolic space, namely that
they do not return to $\tQ$ or $\tS$ and that they approach $\tM$ normally. 
In section \ref{exponential}, we study the geodesic flow equations to extend the exponential map
to the compactified normal bundle and show that the extended map is smooth and a local diffeomorphism.
The extensive debt this paper owes to \cite{maz86} is especially clear here, where
we regularize the flow equations using the method developed there.
The final substantial step, in Section \ref{injectivity}, is to show that the normal exponential map 
is actually injective on a suitably restricted neighborhood
of $\tS$. Many of the previous (and elegant) Cartan-Hadamard-type proofs adapt with difficulty, if at all, to the noncomplete 
and local setting studied here. The homotopy-lifting approach of \cite{her63}, however,
adapts well to this setting, and it enables us to show injectivity. In section \ref{proofs}, the above theorems
and their corollaries follow quickly.

In a sequel, we will use the normal form here constructed to study formal existence and expansion of
cornered asymptotically hyperbolic Einstein metrics, after the manner of \cite{fg12} in the usual case. 

\subsection*{Acknowledgments}
This is doctoral work under the supervision of C. Robin Graham at the University of Washington. I am most grateful to him for suggesting this and related problems, and
for the really extraordinary time and attention he has given to answering questions and making suggestions large and small.
I am also grateful to Andreas Karch for bringing the topic to both of our attention
in the first place, to John Lee and Daniel Pollack for numerous helpful conversations, and to Hart Smith
for financial support. This research was partially supported by the National Science Foundation under RTG Grant 
DMS-0838212 and Grant DMS-1161283.

\section{Cornered Spaces and Blowups}\label{setup}

Recall that a conformally compact manifold is a smooth compact manifold $X$ with boundary $M$, equipped on the interior with a smooth metric $g$
such that, for any defining function $\varphi$ of the boundary $M$, the metric $\varphi^2g$ extends smoothly to a metric on $X$.
The boundary $M$ is called the infinite boundary or boundary at infinity, and if $h = \varphi^2g|_{TM}$ for some defining function
$\varphi$, then the conformal class $[h]$ on $M$ is well-defined and is called the conformal infinity.
A conformally compact manifold is called \emph{asymptotically hyperbolic} if for some (and hence any) such defining function $\varphi$,
we have $|d\varphi|_{\varphi^2g} = 1$ on $M$. The name is due to the fact, shown first in \cite{maz86}, that such manifolds have all sectional curvatures
asymptotic to $-1$.

A natural generalization of a conformally compact manifold is to consider manifolds that have \emph{finite} boundaries as well as the boundary at infinity; a
simple example would be half of the Poincar\'{e} ball. In such spaces, the finite and infinite boundaries meet in a corner, which is at infinity.

We first give an intrinsic definition of this situation.

\begin{definition}
  \label{intrindef}
  A \emph{cornered space} is a smooth manifold with codimension-two corners, $X^{n + 1}$, such that
  \begin{enumerate}[(i)]
    \item There are submanifolds with boundary $M^n \subset \partial X$ and $Q^n \subset \partial X$ of the boundary
      $\partial X$, such that $\emptyset \neq S = M \cap Q$ is the mutual boundary, and is the entire codimension-two corner of $X$, and such that
      $\partial X = M \cup Q$; and
    \item the corner $S \subset M$ is a smooth, compact hypersurface in $M$.
  \end{enumerate}
  We denote a cornered space by $(X,M,Q)$, and we set $\mathring{X} = X \setminus(Q \cup M)$.

  Given a cornered space $(X,M,Q)$, a smooth (resp. $C^k$) \emph{cornered conformally compact metric} on $X$ is a smooth Riemannian metric
  $g_+$ on $X \setminus M$ such that, for any smooth defining function $\varphi$ for $M$, the metric $\varphi^2g_+$ extends
  to a smooth (resp. $C^k$) metric on $X$. We call such a metric a \emph{cornered asymptotically hyperbolic (CAH) metric} if for some
  (hence any) such defining function $\varphi$, the condition $|d\varphi|_{\varphi^2g_+} = 1$ holds along $M$.

  A smooth (resp. $C^k$) \emph{cornered asymptotically hyperbolic (CAH) space} is a cornered space $(X,M,Q)$ together with a smooth (resp. $C^k$)
  CAH metric $g_+$. We will denote such a space by $(X,M,Q,g_+)$. The definition for cornered conformally compact space is analogous.
\end{definition}
For a cornered conformally compact space $(X,M,Q,g_+)$, the \emph{conformal infinity} $[h]$ is the conformal class
$[\varphi^2g_+|_{TM}]$ on $M$, where $\varphi$ is a defining function for $M$. Notice that a consequence of the
fact that $X$ is a manifold with corners
is that the boundary components $M$ and $Q$ intersect transversely.

For each $x \in S$, we define $\theta_0(x)$ to be the angle between $M$ and $Q$ at $X$ with respect to $\varphi^2g_+$, where $\varphi$ is any smooth defining
function for $M$. Plainly $\theta_0 \in C^{\infty}(S)$.

It will be important to our analysis to be able to view $X$ as a submanifold of a larger AH manifold without corner. By doubling across $Q$ (\cite{mel96},
Chapter 1)
and using partitions of unity, we may construct a global AH manifold $(\bX,\bg_+)$ with boundary $\bM$, such that $\mathring{X}$ is an open submanifold
of $\bX$
with $\partial \mathring{X} = M \cup Q$, where $M \subset \bM$ and $Q \subset X$ is a hypersurface in $\bX$, and such
that $\bg_+|_{X \setminus M} = g_+$. The extension $\bg_+$ is not canonical, of course.

As we are planning to study polar-like coordinates at the codimension-two hypersurface $S$, and since such coordinates must be singular there,
we employ the usual measure of blowing up $X$ along $S$ (\cite{mel08}). Let $X$ be a cornered space, with $M$, $Q$, and $S$ as in the definition. 
For $s \in S$, define $N_sS = T_sX / T_sS$, which is a vector space of dimension two. Let $NS$ be the vector
bundle $NS = \sqcup_{s \in S} N_sS$. Let $N_+S \subset NS$ be the inward-pointing normal vectors (including those tangent to $\partial X = Q \cup M$).
Thus $N_+S$ is a bundle with fiber a closed cone in $\mathbb{R}^2$ and base $S$. Finally, let $\tS = (N_+S\setminus \left\{ 0 \right\}) / \mathbb{R}^+$,
which is the total space of a fibration over $S$ with fiber the closed interval $[0,1]$.
Set $\widetilde{X} = (X \setminus S) \sqcup \widetilde{S}$, and define the \emph{blow-down map} $b:\widetilde{X} \to X$ by
$b(x) = x$ ($x \in X \setminus S$) and $b(\tilde{s}) = \pi(\tilde{s})$ ($\tilde{s} \in \widetilde{S}$), 
where $\pi$ is the natural projection.
Then as shown in \cite{mel08}, $\widetilde{X}$ has a unique smooth structure as a manifold with corners of codimension two 
such that $b$ is smooth, $b|_{\tX \setminus \tS}:\tX \setminus \tS \to X \setminus S$ is a diffeomorphism onto its image,
and $db_{\ts}$ has rank $n$ for $\ts \in \tS$. Moreover, polar coordinates on $X$ centered along $S$ lift to smooth coordinates.
We set $\tM = \overline{b^{-1}(M \setminus S)}$ and $\tQ = \overline{b^{-1}(Q \setminus S)}$. Then
$b|_{\tM}:\tM \to M$ and $b|_{\tQ}:\tQ \to Q$ are diffeomorphisms.

Recall that an edge structure on a manifold with boundary is a fibration of the boundary, and the associated edge vector fields are the vector
fields that are tangent to the fibers at the boundary (\cite{maz91}). 
An important special case is a 0-structure (\cite{mm87}), for which the boundary fibers are points
and the edge vector fields are those that vanish at the boundary. On our blowup space $\tX$, the blown-up face $\tS$ is the total space
of the fibration $b|_{\tS}:\tS \to S$ with interval fibers, while we can view $b|_{\tM}:\tM \to M$ as a fibration whose fibers are points. We will
refer to the structure defined by these two fibrations as a 0-edge structure, and the associated 0-edge vector fields are the smooth vector fields
on $\tX$ which are tangent to the fibers at $\tS$, and which vanish at $\tM$.

The 0-edge vector fields may be easily expressed in appropriate local coordinates. Let $\theta$ be a defining function for $\tM$ whose restriction to
each fiber of $\tS$ is a fiber coordinate taking values in $[0,\pi)$; let $\rho$ be any defining function for $\tS$; and locally let
$x^s, 1 \leq s \leq n - 1$, be the lifts to $\tX$ of functions on $X$ that restrict to local coordinates on $S$. Then the vector fields
\begin{equation*}
  \sin\theta \frac{\partial}{\partial \theta}, \quad \rho\sin\theta \frac{\partial}{\partial x^s}, \quad \rho\sin\theta \frac{\partial}{\partial \rho}
\end{equation*}
span the 0-edge vector fields over $C^{\infty}(\tX)$. As in the usual edge case, there is a well-defined vector bundle ${}^{0e}T\tX$
whose smooth sections are the 0-edge vector fields. The smooth sections of the dual bundle ${}^{0e}T^*\tX$ are locally spanned by
\begin{equation}
  \label{dualframe}
  \frac{d\theta}{\sin\theta},\quad\frac{dx^s}{\rho\sin\theta}, \quad\frac{d\rho}{\rho\sin\theta}.
\end{equation}
By a 0-edge metric we will mean a smooth positive definite section $g$ of $S^2({}^{0e}T^*\tX)$. This is equivalent to the condition that
locally $g$ may be written as
\begin{equation*}
  g = \left( \begin{array}{ccc}
    \frac{d\theta}{\sin\theta},&\frac{dx^s}{\rho\sin\theta},&\frac{d\rho}{\rho\sin\theta}\end{array}\right)
    G
  \left(
  \begin{array}{ccc}
    \frac{d\theta}{\sin\theta}\\
    \frac{dx^s}{\rho\sin\theta}\\
    \frac{d\rho}{\rho\sin\theta}
  \end{array}
    \right),
\end{equation*}
where $G$ is a smooth, positive-definite matrix-valued function on $\tX$. This allows us to define the class of metrics that we will study.

\begin{definition}
  \label{admis}
  An admissible metric on $\tX$ is a 0-edge metric $g$ on $\tX$ which can be written in the form
  \begin{equation*}
    g = b^*g_+ + \mathcal{L},
  \end{equation*}
  where $g_+$ is a smooth cornered asymptotically hyperbolic metric on $X$ and $\mathcal{L}$ is a smooth section of
  $S^2({}^{0e}T^*\tX)$ that vanishes on $\tS$ and $\tM$.
\end{definition}
The latter condition is the same as saying that $\mathcal{L} = (\rho\sin\theta)\ell$, for some smooth section $\ell$
of $S^2({}^{0e}T^*\tX)$. We will see below that if $g_+$ is a smooth CAH metric on $X$, then
$b^*g_+$ is a 0-edge metric.

Since $b|_{\tX \setminus (\tM \cup \tS)}:\tX \setminus (\tM \cup \tS) \to X \setminus M$ is a diffeomorphism, an admissible
$g$ uniquely determines a smooth metric $g_X$ on $X \setminus M$ satisfying $b^*g_X = g$ on
$\tX \setminus (\tM \cup \tS)$. Since $\mathcal{L}$ vanishes on $\tS$ and $\tM$, it is not hard
to see that $g_X$ is a $C^0$ CAH metric on $X$. Thus we will call a metric $g_X$ on $X \setminus M$ an admissible metric on $X$ if
$b^*g_X$ extends to an admissible metric on $\tX$.

Observe that an admissible metric $g_X$ on $X$ determines a well-defined angle function $\Theta$ on the blown-up face $\tS$,
which serves as a smooth fiber coordinate.
Let $\ts \in \widetilde{S}$, with $s = b(\ts) \in S$. Then, under one interpretation, 
$\ts$ naturally represents a hyperplane $P_{\ts}$ in $T_sX$
containing $T_sS$. The angle $\Theta(\ts)$ between $P_{s}$ and $T_sM$ is well-defined. 
It can be computed as follows: let $\varphi$ be any defining function for $M$, 
and $\bar{g}_X = \varphi^2g_X$. Let $\bar{\nu}_M \in T_sM$
be normal to $T_sS$, inward pointing in $M$, and unit $\bar{g}_X$-length (this is uniquely defined and continuous,
by the continuity of admissible metrics just observed). 
Similarly, let $\bar{\nu}_{P_{\ts}}$ be inward-pointing in $P_{\ts}$, normal to $T_sS$, and unit length. Then
$\Theta(\ts) = \cos^{-1}(\bar{g}_X(\bar{\nu}_M,\bar{\nu}_{P_{\ts}}))$. We could also have defined $\Theta$ using
$g_+$, and in particular, it is clear that $\Theta \in C^{\infty}(\tS)$.
It is easy to show that this is defined independently of $\varphi$. Thus, $\Theta$ is well-defined.

Let $g_+$ be a smooth CAH metric on $X$.
We construct a product identification on $X$ that we will use extensively throughout, and we then use it
to show that $b^*g_+$ is a 0-edge metric. Choose
an extension $(\bX,\bg_+)$. To each representative $\bh$ on $\bM$ we can associate a neighborhood $\bU$ of $\bM$ in $\bX$ and a unique diffeomorphism
$\chi:[0,\varepsilon)_r \times \bM \to \bU$ such that $\chi|_{\tM} = \id$ and $\chi^*\bg_+ = r^{-2}(dr^2 + \bh_r)$,
where $\bh_0 = \bh$.
Now let $y$ be a geodesic defining function for $S$ in $\bM$ with respect to the metric $\bh$ -- that is, a solution 
near $S$ on $\bM$ to the equation
$|dy|_{\bh}^2 = 1$ with $y|_S \equiv 0$. We choose $y > 0$ on $M$. Then there is 
a diffeomorphism $\psi$ from $S \times (-\delta,\delta)_y$ to a neighborhood $W$ of $S$ in $\bM$
such that $\psi^*\bh = dy^2 + k_y$, 
where $k_y$ is a smooth one-parameter family of metrics on $S$.
Thus, we have shown that there is
a neighborhood $\bU$ of $S$ in $\bX$ and a unique diffeomorphism $\varphi: [0,\varepsilon)_r \times S
\times (-\delta,\delta)_y \to \bU$, 
for which $\varphi|_{\{0\}\times S\times \{0\}} = \id_S$ and
\begin{equation}
\label{cartform}
\varphi^*\bg_+ = \frac{dr^2 + \bh_r}{r^2},
\end{equation}
where $\bh_r$ is a one-parameter family of metrics on $S \times (-\delta,\delta)_y$ with 
\begin{equation}
  \label{g0}
  \bh_0 = dy^2 + k_y.
\end{equation}
We call this the product identification for $\bg_+$ determined by $\bh$, and we let $\pi_S:\bU \to S$ be the projection onto $S$ determined by it.

In cases where $Q$ makes an obtuse angle with $M$, the values inside $X$ of the functions $r$ and $y$ just constructed
will depend on $\bg_+$ outside $X$. We will use the product identification to analyze the behavior of geodesics in $X$,
which of course is independent of the extension chosen.

We obtain smooth coordinates on the blowup $\tX$ near $\tS$ by introducing polar coordinates on $X$. Using the coordinates
defined above, these are given by
\begin{equation}
  \label{polardef}
  r = \rho\sin\theta, \quad y = \rho\cos\theta,\text{ with } \rho \geq 0, \quad 0 \leq \theta < \pi.
\end{equation}
Then locally, a product identification on the
blowup may be given by $p \mapsto (\theta(p),\pi_S(p),\rho(p))$. Observe that
$\widetilde{S}$ is given precisely by $\rho = 0$ and $\widetilde{M}$ is given by $\theta = 0$.
For any admissible metric $g$ on $\tX$ such that $g = b^*g_+ + \mathcal{L}$,
this identification on the blowup will be called a \emph{polar $g$-identification},
or depending on context, \emph{polar $g$-coordinates}. Notice that by (\ref{cartform}) and
(\ref{g0}),
$\theta|_{\tS} = \Theta$.

Now by (\ref{polardef}), we have
\begin{align*}
  dr &= \rho \cos \theta d\theta + \sin\theta d\rho\\
  dy &= -\rho\sin\theta d\theta + \cos\theta d\rho.
\end{align*}
Let $\left\{ x^s \right\}_{s = 1}^{n - 1}$ be local coordinates on $S$. Extend these into $\tX$ near $\tS$
using the product identification. Note by (\ref{g0}) that in (\ref{cartform}), $\bh_r = dy^2 + k_y + O(r)$.
It is then straightforward to compute the metric in our new coordinates:
\begin{equation}
  \label{polformbetter}
  b^*g_+ = \left(\begin{array}{ccc}
    \frac{d\theta}{\sin\theta}, & \frac{dx^s}{\rho \sin\theta}, & \frac{d\rho}{\rho \sin\theta}\end{array}\right)
    G
      \left(\begin{array}{c}
        \frac{d\theta}{\sin\theta} \\ \frac{dx^s}{\rho\sin\theta} \\ \frac{d\rho}{\rho\sin\theta}\end{array}\right),
\end{equation}
where
\begin{equation}
  \label{polformwasbetter}
      G = \left(\begin{array}{ccc}
      1 + O(\rho\sin^3\theta) & O(\rho \sin^2\theta) & O(\rho\sin^2\theta)\\
      O(\rho\sin^2\theta) & k_{\rho\cos\theta} + O(\rho\sin\theta) & O(\rho\sin\theta)\\
      O(\rho\sin^2\theta) & O(\rho\sin\theta) & 1 + O(\rho\sin\theta)\end{array}\right)
\end{equation}
Thus, $b^*g_+$ is a 0-edge metric.
Notice that $k_{\rho\cos\theta} = k_{\rho} + O(\rho \sin^2\theta)$.
This yields the following.
\begin{proposition}
  In a polar identification, an admissible metric $g$ on $\tX$ takes the form
  \begin{equation}
    \label{polform}
    g = \frac{1}{\sin^2(\theta)}\left[ d\theta^2 + \frac{d\rho^2 + k_{\rho}}{\rho^2} \right] + (\rho\sin\theta)\ell,
  \end{equation}
  where $k_{\rho}$ is a one-parameter family of metrics on $S$ and $\ell \in C^{\infty}(S^2({}^{0e}T^*\tX))$.
\end{proposition}
We note that the statement that $g$ can be written in the form (\ref{polform}) is equivalent to the statement that
it can be written as
\begin{equation*}
  g = \frac{1}{\sin^2(\theta)}\left[ d\theta^2 + \frac{d\rho^2 + k_{\theta,\rho}}{\rho^2} \right] + (\rho\sin\theta)\ell,
\end{equation*}
where $\ell$ is as before and where $k_{\theta,\rho}$ is a two-parameter family of metrics on $S$ such that
$k_{\theta,0}$ is independent of $\theta$.

Notice that for the hyperbolic metric, (\ref{hyppolg}) exhibits the form (\ref{polform}) with $k = |dx|^2$ and $\ell = 0$.

It will be useful to have equation (\ref{polform}) expressed in block form. On $\tX$ in the coordinates $(\theta,x^s,\rho)$, the
metric takes the form

\begin{equation}
  \label{g}
  g_{ij} = \csc^2(\theta)\left(\begin{array}{ccc}
    1 + O(\rho\sin\theta) & O(\sin \theta) & O(\sin \theta)\\
    O(\sin \theta) & \rho^{-2}k_{\rho} + O(\rho^{-1}\sin\theta) & O(\rho^{-1}\sin\theta)\\
    O(\sin\theta) & O(\rho^{-1}\sin\theta) & \rho^{-2} + O(\rho^{-1}\sin\theta)
  \end{array}\right).
\end{equation}
This may also be written
\begin{equation*}
  g_{ij} = \csc^2(\theta)A(\rho) \left(\begin{array}{ccc}
      1 + O(\rho\sin\theta) & O(\rho\sin\theta) & O(\rho\sin\theta)\\
      O(\rho\sin\theta) & k_{\rho} + O(\rho\sin\theta) & O(\rho\sin\theta)\\
      O(\rho\sin\theta) & O(\rho\sin\theta) & 1 + O(\rho\sin\theta)\end{array}\right)
      A(\rho),
\end{equation*}
where
\begin{equation*}
  A(\rho) = \left(\begin{array}{ccc}
    1 & &\\
    & \rho^{-1} &\\
    & & \rho^{-1}\end{array}\right).
\end{equation*}
This allows us easily to use Cramer's rule to find that
\begin{equation}
  \label{ginv}
  g^{ij} = \sin^2(\theta)\left(\begin{array}{ccc}
    1 + O(\rho\sin\theta) & O(\rho^2\sin\theta) & O(\rho^2\sin\theta)\\
    O(\rho^2\sin\theta) & \rho^2k_{\rho}^{-1} + O(\rho^3\sin\theta) & O(\rho^3\sin\theta)\\
    O(\rho^2\sin\theta) & O(\rho^3\sin\theta) & \rho^2 + O(\rho^3 \sin\theta)\end{array}\right).
\end{equation}

\subsection{Notation}
Throughout this paper, $(\tX,\tM,\tQ,\tS)$ will be the blowup of a cornered space $(X,M,Q)$, with blowdown map
$b:\tX \to X$. We let $g = b^*g_X = b^*g_+ + (\rho\sin\theta)\ell$ be an admissible metric on $\tX$.
Except where noted otherwise, $\theta$ and $\rho$ will denote the polar coordinates in a polar $g$-coordinate
system. Similarly, $r = \rho\sin\theta$ will denote the geodesic defining function on $X \subset \bX$
in terms of which they were defined. The projection onto the $S$ factor will be denoted $\pi_S$.

$X$ will be of dimension $n + 1$, where unless otherwise specified, $n \geq 2$ always.

We use index notation in polar coordinates. When doing so, $0$ will refer to the
first factor $\theta$, and $n$ will refer to the last factor $\rho$. The indices $1 \leq s,t \leq n - 1$ will
refer to local coordinates on the second factor, $S$, while the indices $0 \leq i, j \leq n$ will run over all
$n + 1$ coordinates. Finally, $1 \leq \mu, \nu \leq n$ will run over $S$ and the last factor.

The metric $g$ will be used to raise and lower indices, except that
$g^{ij}$ is the inverse metric. We write $\bar{g} = \rho^2\sin^2(\theta)g = r^2g$ 
for the metric compactified with respect to polar $g$-coordinates.
Note that $\bar{g}$ is degenerate along $\tS$, as $\left|\frac{\partial}{\partial \theta}\right|_{\bar{g}} = 0$ there.

If $a > 0$, we define
\begin{equation*}
  \tQ_a = \left\{ q \in \tQ: 0 < \rho(q) < a \right\},
\end{equation*}
and
\begin{equation*}
  \ctQ{a} = \left\{ q \in \tQ: 0 \leq \rho(q) < a \right\}.
\end{equation*}
For general open $V \subseteq \tQ$, we define $\underline{V} = V \cup L$, where $L$ is the interior
in $\tQ \cap \tS$ of the set of limit points of $V$ in $\tS$.
We also define
\begin{equation*}
  \tX_a = \left\{ x \in \tX: 0 < \rho(x) < a \right\}.
\end{equation*}

We let $\nu$ be the inward-pointing unit normal vector field on $\tQ \setminus \tS$ with respect to $g$.  If $q \in \tQ \setminus \tS$, we let 
$\gamma_q$ denote the $g$-geodesic that begins at $q$ and has initial tangent vector
$\gamma_q'(0) = \nu_q$.

If $A$ is a covariant $k$-tensor, we write $A = O_g(f)$ to indicate that $|A|_g = O(f)$; or equivalently, that if
$Y_1,\dots,Y_k$ are $g$-unit vector fields, then we have $A(Y_1,\dots,Y_k) = O(f)$ with constant independent of the $Y_i$.
Similarly, if $Y$ is a vector field, we write $Y = O_g(f)$ to indicate $|Y|_g = O(f)$. Note that this condition 
is independent of the particular admissible metric $g$.

\section{Behavior of Normal Geodesics}\label{geods}

As a starting point to our study of the normal exponential map over the finite boundary $\tQ$ of the blowup, 
in this section we study the basic behavior of $g$-geodesics that leave $\tQ$ normally, where $g$ is an admissible metric on
$(\tX,\tM,\tQ,\tS)$.
Throughout, we will work in a polar $g$-identification as constructed in the previous section, and we will let
$\tU$ be a neighborhood of $\tS$ in $\tX$ on which such coordinates exist.

Our first task is to study the normal field to $\tQ \setminus \tS$ near $\tS$. 
\begin{lemma}
  \label{normlem}
  Sufficiently near $\tS$, the normal field $\nu$ on $\tQ \setminus \tS$ satisfies
  \begin{equation}
    \label{qnorm}
    \nu = -\sin\theta \frac{\partial}{\partial \theta} + O_g(\rho).
  \end{equation}
  In particular, $\nu$ extends smoothly to $\tQ \cap \tS$.
\end{lemma}

\begin{proof}
  Using the implicit function theorem and the fact that $\frac{\partial}{\partial \theta}$
  is transverse to $\tQ \cap \tS$,
  we may write $\tQ$ near $\tS$ smoothly as $\theta = \psi(x,\rho)$. Define $f(\theta,x,\rho) = \psi(x,\rho) - \theta$.
  Then using (\ref{ginv}), we find in local coordinates that
  \begin{equation*}
    \grad f = g^{ij}\frac{\partial f}{\partial x^i}\frac{\partial}{\partial x^j} = 
    (-\sin^2(\theta) + O(\rho))\frac{\partial}{\partial \theta} + O(\rho^2),
  \end{equation*}
  while
  \begin{equation*}
    |\grad f|_g = \sin \theta + O(\rho),
  \end{equation*}
  where we keep in mind that $\sin\theta$ is bounded away from $0$ on $\tQ$.
  As $\nu = \frac{\grad f}{|\grad f|_g}$, the result follows.
\end{proof}

It has long been the case that convex functions are important in studying spaces of negative curvature; see \cite{bo69}. Most of our interior analysis of
the $g$-geodesics leaving $Q$ will follow from the fact that the cotangent function on a cornered AH space has a 
Hessian of a very special
form related to convexity. This Hessian equation actually has a history of its own
in the negatively curved setting and beyond; see for example \cite{cc96,hpw15}.

\begin{lemma}
  \label{cotlem}
  Define $w \in C^{\infty}(\tU \setminus \tM)$ by $w = \cot(\theta)$. Then
  \begin{equation}
    \label{convex}
    \nabla_g^2w = wg + O_g(\rho).
  \end{equation}
\end{lemma}
\begin{remark}This result is motivated by the fact that, in the case of the hyperbolic upper half-space, the equation holds exactly.\end{remark}
\begin{proof}
We will need the Christoffel symbols $\Gamma_{ij}^0$. Computing from (\ref{g}) and (\ref{ginv}) using
the equation $\Gamma_{ij}^k = \frac{1}{2}g^{kl}(\partial_ig_{lj} + \partial_jg_{li} - \partial_lg_{ij})$,
we find that
\begin{equation*}
  \Gamma^0_{ij} = \left(\begin{array}{ccc}
    -\cot(\theta) + O(\rho) & O(1) & O(1)\\
    O(1) & \rho^{-2}\cot(\theta)k_{st} + O(\rho^{-1}) & O(\rho^{-1})\\
    O(1) & O(\rho^{-1}) & \rho^{-2}\cot(\theta) + O(\rho^{-1})\end{array}\right).
\end{equation*}
Now $dw = -\csc^2(\theta)d\theta$, and we can use these Christoffel computations to find that
\begin{align*}
  \nabla^2w = \nabla(dw) &= \cot(\theta)g + \frac{\left(
  d\theta, dx^s, d\rho\right)}{\sin^2\theta}\left(\begin{array}{ccc}
    O(\rho) & O(1) & O(1)\\
    O(1) & O(\rho^{-1}) & O(\rho^{-1})\\
    O(1) & O(\rho^{-1}) & O(\rho^{-1})\end{array}\right)
    \left(\begin{array}{c}d\theta \\ dx^t \\ d\rho\end{array}\right)\\
      &= \cot(\theta)g\\ 
      &\quad+\left( \frac{d\theta}{\sin\theta}, \frac{dx^s}{\rho\sin\theta}, \frac{d\rho}{\rho\sin\theta} \right)
      \left(\begin{array}{ccc}O(\rho) & O(\rho) & O(\rho)\\
        O(\rho) & O(\rho) & O(\rho)\\
        O(\rho) & O(\rho) & O(\rho)\end{array}\right)
        \left(\begin{array}{c}
          \frac{d\theta}{\sin\theta}\\
          \frac{dx^t}{\rho\sin\theta}\\
          \frac{d\rho}{\rho\sin\theta}\end{array}\right).
\end{align*}
This proves the claim.
\end{proof}

We next need two technical results.

\begin{proposition}
  \label{complem}
  Let $D = J \times \mathbb{R}_{(u,v)}^2 \subseteq \mathbb{R}^3$, where $J$ is an interval containing $[a,b)$.
  Let $f:D \to \mathbb{R}$ be continuous and locally Lipschitz,
  and weakly increasing in $u$. Define an ordinary differential operator
  $L$ by $Lu(x) = u''(x) - f(x,u(x),u'(x))$. Suppose $\theta, \psi:[a,b) \to \mathbb{R}$ are $C^2$ functions
  such that the graphs of $t \mapsto (\theta(t),\theta'(t))$ and $t \mapsto (\psi(t),\psi'(t))$ lie in
  $D$. If
  \begin{itemize}
    \item $\theta(a) \leq \psi(a)$; and
    \item $\theta'(a) \leq \psi'(a)$; and
    \item $L\theta \leq L\psi$ on $[a,b)$,
  \end{itemize}
  then $\theta(t) \leq \psi(t)$ and $\theta'(t) \leq \psi'(t)$ for all $t \in [a,b)$.
\end{proposition}
This is essentially Theorem 11.XVI in \cite{wal98}.

\begin{lemma}
  \label{lemcsc}
  (a) Let $0 < \delta < 1$, $a > 0$, and $b \in \mathbb{R}$,
  and set $f(t) = a e^t + b e^{-t} + \delta$. Then there exists a continuous
  function $C = C(a,b,\delta) > 0$
  such that, if $w(t) \geq f(t)$ for all $t \geq 0$, then $1 + w(t)^2 \geq C^{-2}e^{2t}$ for all $t \geq 0$. 

  (b) Let $0 < \delta < 1$, $0 < a_1 < a_2$, and $b_1, b_2 \in \mathbb{R}$. Suppose that
  $a_1e^t + b_1e^{-t} + \delta \leq w(t) \leq a_2e^t + b_2e^{-t} - \delta$. Then
  there exists $D = D(a_1,a_2,b_1,b_2,\delta)$, continuous in its arguments, such that
  $1 + w(t)^2 \leq D^2e^{2t}$ for all $t \geq 0$.
\end{lemma}
\begin{proof}
  (a) There exists $T > 0$ such that, for all $t \geq T$, we have $f(t) > \frac{1}{2}a e^t$.
  Let $\frac{2}{a} < C$ be such that $C^{-2}e^{2T} \leq 1$. The result follows immediately.

  The proof of (b) is similar.
\end{proof}

We also recall the following result from \cite{maz86}.

\begin{proposition}[Propositions 1.8 and 1.9 in \cite{maz86}]
  \label{mazprop}
  Let $(X,M,g)$ be an asymptotically hyperbolic manifold, and $\varphi$ a defining function for $M$. There exists $\varphi_0 > 0$ such that,
  if $p \in X$ with $0 < \varphi(p) < \varphi_0$, and if $\gamma:[0,\infty) \to
    \mathring{X}$ is a geodesic ray with $\gamma(0) = p$ and
  $(\varphi\circ\gamma)'(0) < 0$, then $\gamma$ asymptotically approaches a well-defined point of $M$, normally with
  respect to $\varphi^2g$.
\end{proposition}

In the introduction, we discussed a subset $(X,M,Q)$ of hyperbolic space 
$\mathbb{H}^{n + 1}$ as an example of a CAH manifold,
with $Q$ a Euclidean plane.
Recall that geodesics leaving
$Q$ normally are semi-circles that approach the infinite boundary $\mathbb{R}^n$ orthogonally. The following theorem, which is the 
main result of this section, shows that this behavior is approximated by the geodesics in a general cornered AH space.

First some notation. We 
let $d_S$ be the distance function on $S$ with respect
to $k_0$, where $k_0$ is as in (\ref{g}), and $\pi_S:\tU \to S$ be projection onto the $S$ factor in the polar 
$g$-identification.

\begin{proposition}
  \label{geodprop}
  Let $(\tX,\tM,\tQ,\tS)$ be the blowup of the cornered space $(X, M, Q)$, and $g$ an admissible metric. There exists $a > 0$ such that for each $q \in \tQ_a$,
  \begin{itemize}
    \item $\gamma_q$ exists for all $t \geq 0$ and $\gamma_q(t) \in \mathring{\tX}$ for all $t > 0$;
    \item the limit $\lim_{t \to \infty}\gamma_q(t)$ exists and lies in $\tM \setminus \tS$, and $\gamma_q$ approaches $\tM$
      $\bar{g}$-normally; and
    \item for all $t \geq 0$, $(\theta \circ \gamma_q)'(t) < 0$.
  \end{itemize}
  Moreover, there exist $\varepsilon > 0$ and $A_1, A_2, C > 0$ such that, for all $q \in \tQ_a$ and all $t \in [0,\infty)$,
  \begin{enumerate}
    \item\label{rhocond} $\varepsilon \rho(q) < \rho(\gamma_q(t)) < C\rho(q)$;
    \item $d_S(\pi_S(q),\pi_S(\gamma_q(t))) < C\rho(q)$; and\label{Scond}
    \item $A_1e^{-t} < \sin(\theta(\gamma(t))) < A_2e^{-t}$.\label{thetacond}
  \end{enumerate}
\end{proposition}

Note that we will improve this result in Proposition \ref{extendprop}; in particular, it will imply that
in \ref{Scond} we could write $\rho(q)^2$ instead of $\rho(q)$, and that in \ref{rhocond} we can take
$\frac{\varepsilon}{C} \to 1$ as $a \to 0$.

\begin{proof}
  By Lemma \ref{normlem}, $\nu = -\sin(\theta)\frac{\partial}{\partial \theta} + O_g(\rho)$. 
  Set $\nu_q^{\theta} = d\theta(\nu_q)$. Thus, for $q \in \tQ \setminus \tS$ sufficiently near
  $\tS$, we have
  $-\csc^2(\theta(q))\nu^{\theta}_q = \csc(\theta(q)) + O(\rho(q))$, uniformly in $q$. 
  Now for $\theta \in (0,\pi)$, we always have $\csc\theta > |\cot\theta|$.
  Let $0 < \delta < 1$ be such that there exists $\rho_0 > 0$ so that $\tX_{\rho_0} \subset \tU$ and
  such that
  \[\alpha := \inf_{q \in \tQ_{\rho_0}}(-\csc^2(\theta(q))\nu_q^{\theta} - |\cot(\theta(q))| - \delta) > 0.\]
  Such a $\delta$ exists because $\theta$ is bounded away from $0$ and $\pi$ on $\tQ$.
  Now let
  \begin{equation*}
    \begin{split}
    \beta :=& \sup_{q \in \tQ_{\rho_0}}\left(\max\left\{ -\csc^2(\theta(q))\nu_q^{\theta} + |\cot(\theta(q))| + \delta,\right.\right.\\
  &\qquad\left.\left.|\csc^2(\theta(q))\nu_q^{\theta} + \cot(\theta(q)) + \delta|\right\}\right) > 0,
\end{split}
\end{equation*}
  which is finite because the cosecant and cotangent functions are bounded on $\tQ$ near $\tS$.
  Next, let $B > 0$ be large enough that, for all $Y \in T(\tX_{\rho_0})$ with $|Y|_g = 1$, we have
  $|d\rho(Y)| < B\rho\sin\theta$; such $B$ exists by (\ref{g}).

  For $q \in \tQ_{\rho_0}$, define $\theta_q(t) = \theta(\gamma_q(t))$, and $\rho_q(t) = \rho(\gamma_q(t))$.
  Next, define $w_q(t) = \cot(\theta_q(t))$; thus $w_q$ is defined on the same domain as
  the geodesic $\gamma_q$. Noting that $w_q(0) = \cot(\theta(q))$ and $\dot{w}_q(0) = -\csc^2(\theta)\nu_q^{\theta}$, let $C$ be as defined
  in Lemma \ref{lemcsc} and set
  \begin{equation*}
    A := \sup_{q \in \tQ_{\rho_0}} C\left( \frac{1}{2}(w_q(0) + \dot{w}_q(0) - \delta),
    \frac{1}{2}(w_q(0) - \dot{w}_q(0) - \delta),\delta\right) > 0,
  \end{equation*}
  which is finite. Also let $D$ be as in Lemma \ref{lemcsc} and set
  \begin{equation*}
    \begin{aligned}
      E &:= \sup_{q \in \tQ_{\rho_0}}D\left( \frac{1}{2}(w_q(0) + \dot{w}_q(0) - \delta),\frac{1}{2}(w_q(0) + \dot{w}_q(0) + \delta),\right.\\
    &\qquad\frac{1}{2}(w_q(0) - \dot{w}_q(0) - \delta),
    \left.\frac{1}{2}(w_q(0) - \dot{w}_q(0) + \delta),\delta\right) > 0,
  \end{aligned}
  \end{equation*}
  which is likewise finite.
  By shrinking it if necessary, we may assume that $\rho_0$ is small enough that, if $q \in \tQ_{\rho_0}$, then
  \begin{equation}
    \label{nubd}
    \left|\nu_q + \sin(\theta)\frac{\partial}{\partial \theta}\right|_g < \frac{\alpha E^{-1}}{8}.
  \end{equation}
  We may similarly suppose, by (\ref{g}), that on $\tQ_{\rho_0}$,
  \begin{equation}
    \label{g00bd}
    \sin(\theta)|g_{00} - \csc^2(\theta)| < \min\left\{\frac{\alpha E^{-1}}{8A\beta},1\right\},
  \end{equation}
  and that if $Y \in T\tX|_{\tQ_{\rho_0}}$ with $d\theta(Y) = 0$, then
  \begin{equation}
    \label{g0mubd}
    \left|\left\langle\sin\theta\frac{\partial}{\partial \theta},Y\right\rangle\right| \leq \frac{\alpha E^{-1}}{8(1 + \sqrt{2}A\beta)}|Y|_g.
  \end{equation}

  Now because $\gamma_q$ is a geodesic, $\frac{d^2}{dt^2}w(\gamma_q(t)) = (\nabla_g^2w)(\dot{\gamma}_q,\dot{\gamma}_q)$.
  It follows by Lemma \ref{cotlem} that
  \begin{equation*}
    \ddot{w}_q = w_q + O(\rho_q(t)).
  \end{equation*}
  By shrinking $\rho_0$ if necessary, we assume that the $O(\rho)$ term in this equation is bounded by $\delta$ for
  $\rho \leq \rho_0$. Now let $0 < a < \frac{1}{2e^{AB}}\rho_0$. We henceforth assume $q \in \tQ_{a}$.

  Now let $f_{\pm}$ be the solutions to $\ddot{f}_{\pm} = f_{\pm} \pm \delta$, with $f_{\pm}(0) = w_q(0)$ and
  $\dot{f}_{\pm}(0) = \dot{w}_q(0)$. Then
  \begin{equation*}
    f_{\pm}(t) = \frac{1}{2}\left( w_q(0) + \dot{w}_q(0) \pm \delta \right)e^{t} + \frac{1}{2}
    \left( w_q(0) - \dot{w}_q(0) \pm \delta \right)e^{-t} \mp \delta.
  \end{equation*}
  The leading coefficient is always positive, by our choice of $\delta$. Moreover, we have
  \begin{equation*}
    \ddot{f}_- - f_- = -\delta \leq \ddot{w}_q - w_q \leq \delta = \ddot{f}_+ - f_+,
  \end{equation*}
  so by Proposition \ref{complem}, we have
  \begin{equation}
    \label{fineq}
    f_-(t) \leq w_q(t) \leq f_+(t)
  \end{equation}
  for all $t \geq 0$ such that $\rho_q(t) < \rho_0$ up to $t$, and so long as the geodesic continues to exist. Also by the same proposition,
  \begin{equation}
    \label{dotineq}
    \dot{f}_-(t) \leq \dot{w}_q(t) \leq \dot{f}_+(t),
  \end{equation}
  subject to the same constraints. Since both bounding functions are positive, we conclude that $\dot{w}_q(t) > 0$
  for all $q$, and for all $t \geq 0$ such that $\rho_q(t) < \rho_0$. This implies that $\dot{\theta} < 0$ for all such $t \geq 0$.
  In addition, the coefficients
  appearing in $f_{\pm}$ are uniformly bounded in $q$. It follows that we have shown that
  $\theta_q(t)$ goes to zero and $\cot(\theta_q(t))$ goes to infinity exponentially, \emph{so long as}
  $\rho_q(t)$ remains bounded by $\rho_0$ and $\gamma_q$ exists.

  Now by Lemma \ref{lemcsc} and the definition of $f_-$, we have
  \begin{equation*}
    \csc^2(\theta_q(t)) = 1 + w_q(t)^2 \geq A^{-2}e^{2t}.
  \end{equation*}
  (This and the following continue, for now, to depend on the assumption that $\rho$ remains bounded by $\rho_0$.) 
  Hence, also,
  \begin{equation}
    \label{sinleq}
    \sin(\theta_q(t)) \leq Ae^{-t}.
  \end{equation}
  Also by (\ref{fineq}), by definition of $E$, and by Lemma \ref{lemcsc}, we have
  \begin{equation}
    \label{singeq}
    \sin(\theta_q(t)) \geq E^{-1}e^{-t}.
  \end{equation}

  It now follows from (\ref{sinleq}) that $\left|\frac{\dot{\rho}_q}{\rho_q}\right| < ABe^{-t}$, by definition of $B$. Hence, at least as long as
  $\rho \leq \rho_0$, we find by integrating that
  \begin{equation}
    \label{rhoineq}
    e^{-AB}\rho_q(0) < \rho_q(t) < e^{AB}\rho_q(0).
  \end{equation}
  But then, since $\rho_q(0) \leq a \leq \frac{\rho_0}{2e^{AB}}$, $\rho_q(t)$ must remain bounded by $\frac{\rho_0}{2}$; so a brief contradiction
  argument shows that, indeed, (\ref{fineq}) -- (\ref{rhoineq}) hold for all time $t \geq 0$ such that $\gamma_q$ exists. 
  We have also shown that
  $\gamma_q$ remains bounded away from $\tS$, i.e., $\rho$ is bounded away from $0$. This shows that $\gamma_q$ never reaches $\tS$, and also, with
  (\ref{rhoineq}), yields \ref{rhocond}. Also, by (\ref{sinleq}) and (\ref{singeq}), we have (c).

  We next analyze the motion in the tangential directions along $S$. It follows from the definition and smoothness
  of $k_{\rho}$ in (\ref{g}) that there exists $K > 0$ such that,
  for unit-length 
  $Y \in T\tX_{\rho_0}$, we have
  $|d\pi_S(Y)|_{k_{0}} < K\rho\sin(\theta)$. It then follows, using (\ref{sinleq}) and (\ref{rhoineq}), that
  \begin{equation*}
    \int_0^{\infty} |d\pi_S(\dot{\gamma}(t))|_{k_0} dt < AKe^{AB}\rho(q),
  \end{equation*}
  which yields \ref{Scond}.
  
  Since $\cot(\theta)$ eventually becomes positive with $\dot{\theta}$ negative, we conclude that $\dot{r}$ is ultimately negative, and so the fact that
  $\gamma_q$ approaches a defined point of $\tM$ normally, if it exists for all time, 
  follows by the analysis of geodesics in the standard AH case, Proposition \ref{mazprop}. 
  Thus, we have established all desired behavior,
  except that the geodesic $\gamma_q$ might leave $\mathring{\tX}$ and return to $\tQ$, ceasing to exist. Since the geodesic is unit speed,
  it either exists for all time or returns to $\tQ$ in finite time, and so we have only to show that the latter does not happen.

  Suppose by way of contradiction that $\gamma_q$ does return to $\tQ$, say at $q'$ and at time $t_1 > 0$. Then $\rho(q') < \rho_0$, and we have
  $\langle \dot{\gamma}_q(t_1),\nu_{q'}\rangle \leq 0$. Now by (\ref{dotineq}), the definitions of $f_-$ and $\alpha$, and (\ref{singeq}), we deduce that
  \begin{equation}
    \label{tdotlow}
    \frac{\alpha E^{-1}}{2} \leq \frac{|\dot{\theta}_q(t)|}{\sin(\theta_q(t))}
  \end{equation}
  for all times $t \geq 0$, and in particular
  $t_1$. Similarly, by (\ref{dotineq}), (\ref{sinleq}), the definition of $f_+$, and the definition of $\beta$, we get
  \begin{equation}
    \label{tdothigh}
    \frac{|\dot{\theta}_q(t)|}{\sin(\theta_q(t))} \leq A\beta.
  \end{equation}
  By (\ref{tdothigh}) and (\ref{g00bd}), and because $|\sin\theta| \leq 1$, we find that at $t_1$,
  \begin{align*}
    \left\langle\dot{\theta}\frac{\partial}{\partial \theta},\dot{\theta}\frac{\partial}{\partial \theta}\right\rangle &=
    \frac{\dot{\theta}^2}{\sin^2\theta} + (g_{00} - \csc^2(\theta))\dot{\theta}^2\\
    &\leq A^2\beta^2 + A^2\beta^2 = 2A^2\beta^2.
    \intertext{Thus,}
    \left|\dot{\theta}\frac{\partial}{\partial \theta}\right|_g &\leq \sqrt{2}A\beta.\numberthis\label{thetaveclen}
  \end{align*}
  
  By (\ref{nubd})-(\ref{g0mubd}), (\ref{tdotlow})-(\ref{thetaveclen}), Cauchy-Schwartz, and
  the triangle inequality, we find that
  \begin{align*}
    \langle \dot{\gamma}_q(t_1),\nu_{q'}\rangle &= \left\langle \dot{\gamma}_q(t_1),-\sin(\theta)\frac{\partial}{\partial \theta}\right\rangle
    + \left\langle \dot{\gamma}_q(t_1),\nu_{q'} + \sin(\theta)\frac{\partial}{\partial \theta}\right\rangle\\
    &= \frac{|\dot{\theta}_q(t_1)|}{\sin(\theta)} - (g_{00} - \csc^2(\theta))\sin(\theta)\dot{\theta}_q(t_1)\\
    &\qquad+ 
    \left\langle \dot{\gamma}_q(t_1) - 
    \dot{\theta}_q(t_1)\frac{\partial}{\partial \theta},
    -\sin(\theta)\frac{\partial}{\partial \theta}\right\rangle
    + \left\langle\dot{\gamma}_q(t_1),\nu_{q'} + \sin(\theta)\frac{\partial}{\partial \theta}\right\rangle\\
    &\geq \frac{\alpha E^{-1}}{2} - \frac{\alpha E^{-1}}{8A\beta}(A\beta) - \frac{\alpha E^{-1}}{8(1 + \sqrt{2}A\beta)}
    \left( 1 + \left|\dot{\theta}\frac{\partial}{\partial\theta}\right|_g \right)\\
    &\qquad- |\dot{\gamma}_q|_g\cdot \left|\nu_q + \sin(\theta)\frac{\partial}{\partial\theta}\right|_g\\
    &\geq \frac{\alpha E^{-1}}{8} > 0,
  \end{align*}
  which is a contradiction. Hence, as desired, $\gamma_q$ does not return to $\tQ$.
\end{proof}

\section{The Exponential Map}\label{exponential}
We continue our analysis of the geodesics leaving $\tQ$ normally now by turning our attention to the mapping properties of the exponential map
on the normal bundle to $\tQ$. We ultimately must prove that this map is a diffeomorphism on a suitable space. For now, we content ourselves
with more local properties.

Let $N_+(\tQ \setminus \tS)$ be the inward-pointing half-closed normal 
\emph{ray} bundle to $\tQ \setminus \tS$, so that $N_+(\tQ \setminus \tS) \approx
[0,\infty)_t \times (\tQ \setminus \tS)$ by the identification $t\nu_q \mapsto (t,q)$; and similarly for the normal bundle over subsets of
$\tQ \setminus \tS$. We denote the normal exponential map by $\exp$. 
We have shown in Proposition \ref{geodprop} that there is some $a > 0$ such that $\exp$ is defined on the entirety of
$N_+\tQ_a$, and takes its values in $\tX \setminus (\tM \cup \tS)$. Trivially, $\left\{ 0 \right\} \times
(\tQ \setminus \tS)$ is mapped by $\exp$ to $\tQ \setminus \tS$ as the identity,
and Proposition \ref{geodprop} also shows that $\exp|_{N_+\tQ_a}^{-1}(\tQ_a) = \left\{ 0 \right\} \times 
\tQ_a$ as well. In order to show that
the exponential map induces a diffeomorphism with a neighborhood of $\tS$, we will have to analyze it as
$q \to \tS$ and as $t \to \infty$. 
We thus introduce a partial compactification of the normal bundle that includes faces corresponding to $t = \infty$ and to
$[0,\infty] \times (\tQ \cap \tS)$, and we will show that the exponential map is defined and a local diffeomorphism on the entire space.

Let $V$ be a neighborhood of $\tQ \cap \tS$ in $\tQ$. Then as observed previously,
$N_+(V\setminus\tS)$ has a natural identification, induced by $\nu$, with $[0,\infty) \times (V \setminus \tS)$. Letting $t$ be the coordinate on
the first factor, we set $\tau = 1 - e^{-t}$, and hence obtain an identification with $[0,1) \times (V \setminus \tS)$. We thus define the compactification
$\cbundle{N_+(V\setminus \tS)} = [0,1] \times V$, and we regard $N_+(V\setminus\tS) \subset 
\cbundle{N_+(V \setminus \tS)}$ 
as a subspace via the identification just described. Note that we have added two new faces in this compactification:
one corresponding to $t = \infty$, and one corresponding to $[0,1] \times (\tQ\cap\tS)$. We will consistently let
$\tau$ be the coordinate on the first factor of $\cbundle{N_+(V \setminus \tS)}$. The space 
$\cbundle{N_+(V\setminus \tS)}$ has a natural smooth structure as a manifold with
corners, and $T\cbundle{N_+(V \setminus S)} \cong T[0,1] \oplus TV$ canonically. 
We note that $\cbundle{N_+(V\setminus \tS)}$ is not quite a compactification, since the interior boundary
of $V$ in $\tQ$ is still not included.

With the compactification of the normal bundle in hand, we are ready to extend the exponential map to reach the boundary. In proving
the following, we follow the approach in \cite{maz86}. For the statement, notice that
$\theta \mapsto v(\theta) := \csc\theta - \cot\theta$ is a diffeomorphism of $(0,\pi)$ with $(0,\infty)$.

\begin{proposition}
  \label{extendprop}
  There exists $\rho_0 > 0$ such that the exponential map $\exp:N_+\tQ_{\rho_0} \to \mathring{\tX}$ extends smoothly to
  a map $\exp:\cbundle{N_+\tQ_{\rho_0}} \to \tX$, and the extended map is a local diffeomorphism of manifolds with
  corners. For $q \in \tQ \cap \tS$, $\exp$ maps $[0,1] \times \left\{ q \right\}$ to the $b$-fiber of
  $\tS$ containing $q$.
  For such $q$, $\exp$ satisfies
  \begin{equation}
    \label{ssoln}
    v(\Theta(\exp(\tau,q))) = v(\Theta(q))(1 - \tau);
  \end{equation}
  that is, in the $v$ coordinate, $\exp$ is a linear function of $\tau$.

  Moreover, for $1 \leq \mu \leq n$ and $q \in \tQ$ and for any $\tau$, the equation
  \begin{equation}
    \label{quadest}
    x^{\mu}(\exp(\tau,q)) = x^{\mu}(q) + O(\rho(q)^2)
  \end{equation}
  holds uniformly in $\tau \in [0,1]$.

  Finally, there exists $c > 0$ such that, if $Y \in T_q\tQ_{\rho_0} \subset T[0,\infty) \oplus T\tQ_{\rho_0}
  \cong TN_+\tQ_{\rho_0}$ and $t\nu_q \in N_+\tQ_{\rho_0}$, then
  \begin{equation}
    \label{ceq}
    |d\exp_{t\nu_q}(Y)|_{\bar{g}} \geq c|Y|_{\bar{g}}.
  \end{equation}
\end{proposition}
\begin{proof}
  We begin by recalling the equations for the geodesic flow on the cotangent bundle. Let $\left\{ x^s \right\}$ be local coordinates for $S$,
  so that $(\theta,x^s,\rho)$ are coordinates on some neighborhood $\tU \subseteq \tX$ of a fiber $F$ in $\tS$. 
  Let $\rho_1$ be small enough that Proposition \ref{geodprop} holds on $\tQ_{\rho_1}$, and let
  $V \subseteq \ctQ{\rho_1}$ be a sufficiently small neighborhood of the point $\tQ \cap F$ that
  normal geodesics off points in $V \setminus \tS$ remain in $\tU$.
  
  The geodesic flow off points of $V \setminus \tS$ then satisfies
  \begin{align*}
    \dot{x}^i &= g^{ij}\xi_j\\
    \dot{\xi}_i &= -\frac{1}{2}\frac{\partial g^{kl}}{\partial x^i}\xi_k\xi_l.
  \end{align*}
We also have
\begin{equation}
  \label{normedeq}
  g^{ij}\xi_i\xi_j = 1.
\end{equation}
We use this fact to rewrite the geodesic equations in terms of $\bar{g} = \rho^2\sin^2(\theta)g$ or, rather, $\bar{g}^{-1} = \rho^{-2}\csc^2(\theta)g^{-1}$, obtaining
\begin{equation}
  \label{floweq1}
\begin{aligned}
  \dot{x}^i &= \rho^2\sin^2(\theta)\bar{g}^{ij}\xi_j\\
  \dot{\xi}_i &= -\frac{1}{2}\frac{\partial}{\partial x^i}\left[ \rho^2\sin^2(\theta)\bar{g}^{kl} \right]\xi_k\xi_l\\
  &= -\frac{\rho_i}{\rho} - \cot(\theta)\theta_i - \frac{1}{2}\rho^2\sin^2(\theta)\frac{\partial \bar{g}^{kl}}{\partial x^i}\xi_k\xi_l.
\end{aligned}
\end{equation}
This system is obviously degenerate at both $\theta = 0$ and $\rho = 0$. We thus introduce rescaled variables, setting
\begin{equation}
  \label{rescaled}
  \begin{aligned}
  \bar{\xi}_{\mu} &= \rho \xi_{\mu}\quad(1 \leq \mu \leq n)\\
  \bar{\xi}_0 &= \sin(\theta)\xi_0.
\end{aligned}
\end{equation}
Hence,
\begin{equation}
  \label{eqmotresc}
  \begin{aligned}
  \dot{\bar{\xi}}_{\mu} &= \dot{\rho}\xi_{\mu} + \rho\dot{\xi}_{\mu} = \frac{\dot{\rho}}{\rho}\bar{\xi}_{\mu} + \rho\dot{\xi}_{\mu}\\
  \dot{\bar{\xi}}_0 &= \cos(\theta)\dot{\theta}\xi_0 + \sin(\theta)\dot{\xi}_0 = \cot(\theta)\dot{\theta}\bar{\xi}_0 + \sin(\theta)\dot{\xi}_0.
\end{aligned}
\end{equation}
Now
\begin{equation}
  \label{dottedeq}
  \begin{aligned}
  \dot{\rho} &= \rho^2\sin^2(\theta)\bar{g}^{nj}\xi_j = \rho\sin^2(\theta)\bar{g}^{n\mu}\bar{\xi}_{\mu} + \rho^2\sin(\theta)\bar{g}^{n0}\bar{\xi}_0\text{ and}\\
  \dot{\theta} &= \rho^2\sin^2(\theta)\bar{g}^{0j}\xi_j = \rho\sin^2(\theta)\bar{g}^{0\mu}\bar{\xi}_{\mu} + \rho^2\sin(\theta)\bar{g}^{00}\bar{\xi}_0.
\end{aligned}
\end{equation}
Thus, rewriting our equations of motion (\ref{floweq1}) and (\ref{eqmotresc}) in terms of our new variables, we get
\begin{equation}
  \label{floweq2}
  \begin{aligned}
    \dot{x}^i &= \rho\sin^2(\theta)\bar{g}^{i\mu}\bar{\xi}_{\mu} + \rho^2\sin(\theta)\bar{g}^{i0}\bar{\xi}_0\\
    \dot{\bar{\xi}}_{\mu} &= (\sin^2(\theta)\bar{g}^{n\nu}\bar{\xi}_{\nu} + \rho\sin(\theta)\bar{g}^{n0}\bar{\xi}_0)\bar{\xi}_{\mu} - \rho_{\mu} -
    \frac{1}{2}\rho\sin^2(\theta)\frac{\partial \bar{g}^{\sigma\lambda}}{\partial x^{\mu}}\bar{\xi}_{\sigma}\bar{\xi}_{\lambda}
    \\&\qquad- \rho^2\sin(\theta)\frac{\partial \bar{g}^{0\sigma}}{\partial x^{\mu}}\bar{\xi}_0\bar{\xi}_{\sigma} - \frac{1}{2}\rho^3\frac{\partial\bar{g}^{00}}
    {\partial x^{\mu}}\bar{\xi}_0^2\\
    \dot{\bar{\xi}}_0 &= (\rho\sin(\theta)\cos(\theta)\bar{g}^{0\mu}\bar{\xi}_{\mu} + \rho^2\cos(\theta)\bar{g}^{00}\bar{\xi}_0)\bar{\xi}_0 - \cos(\theta)
    - \frac{1}{2}\sin^3(\theta)\frac{\partial \bar{g}^{\sigma\lambda}}{\partial \theta}\bar{\xi}_{\sigma}\bar{\xi}_{\lambda}\\&\qquad
    -\rho\sin^2(\theta)\frac{\partial \bar{g}^{0\sigma}}{\partial \theta}\bar{\xi}_0\bar{\xi}_{\sigma}-\frac{1}{2}\rho^2\sin(\theta)
    \frac{\partial \bar{g}^{00}}{\partial \theta}\bar{\xi}_0^2.
  \end{aligned}
\end{equation}
Now $\bar{g}^{00} = O(\rho^{-2})$; otherwise $\bar{g}^{-1}$ is smooth on $\tX$ by (\ref{ginv}). It follows that
equations (\ref{floweq2}) have smooth coefficients all the way up to $\rho = 0$. 

We already know that, for $\rho(q) \neq 0$, solutions exist for all time $t$. We turn to study the case
when $\rho(q) = 0$, that is, when the geodesic starts from some point $q \in V \cap \tS$. Using (\ref{ginv}), and
recalling that $\bar{g}_{st}$ is independent of $\theta$ at $\rho = 0$, note that when $\rho = 0$, the equations (\ref{floweq2}) are given by
\begin{equation}
  \label{rho0eq}
\begin{aligned}
  \dot{\theta} &= \sin(\theta)\bar{\xi}_0\\
  \dot{x}^{\mu} &= 0\\
  \dot{\bar{\xi}}_0 &= \cos(\theta)\bar{\xi}_0^2 - \cos(\theta)\\
  \dot{\bar{\xi}}_{\mu} &= \sin^2(\theta)\bar{\xi}_n\bar{\xi}_{\mu} - \rho_{\mu} + \rho_{\mu}\bar{\xi}_0^2.
\end{aligned}
\end{equation}
Our initial conditions for $q$, by (\ref{qnorm}), (\ref{normedeq}), and the above, are
\begin{equation}
  \label{flowindcond}
  \begin{aligned}
  \theta(0) &= \theta(q)\\
  x^s(0) &= x^s(q)\\
  \rho(0) &= 0\\
  \bar{\xi}_0(0) &= -1\\
  \bar{\xi}_{\mu}(0) &= 0.
\end{aligned}
\end{equation}
Let $\psi(v)$ be the inverse of the function $\theta \mapsto \csc \theta - \cot \theta$; thus $\psi$ is defined on 
$(0,\infty)$ and is monotonic increasing from
$0$ to $\pi$. Then observe that the solution to (\ref{floweq2}) with the given initial conditions is given by
\begin{equation}
  \label{rho0sol}
  \begin{aligned}
    \theta(t) &= \psi((\csc\theta(q) - \cot\theta(q))e^{-t})\\
    x^s(t) &\equiv x^s(0)\\
    \rho(t) &\equiv 0\\
    \bar{\xi}_0(t) &\equiv -1\\
    \bar{\xi}_{\mu}(t) &\equiv 0.
  \end{aligned}
\end{equation}
This exists for all $t \geq 0$, and it satisfies the properties that $\dot\theta < 0$ for all time and that $\lim_{t \to \infty}\theta(t) = 0$. 
Using smooth dependence of solutions on initial conditions, we conclude that
solutions to the 
geodesic equations may be smoothly extended to $\rho = 0$ for all time $t \geq 0$. Thus, $\exp$ is smooth on 
$[0,\infty) \times V$; and by compactness of $\tS$, on $[0,\infty)_t \times \ctQ{\rho_0}$ for some $\rho_0$.

We now turn our attention to $\theta = 0$, which corresponds to $t = \infty$.
We compactify the normal bundle, as above, by setting $\tau = 1 - e^{-t}$, and we wish to show that the exponential map is smooth to $\tau = 1$. 
It will be important throughout to understand the asymptotic behavior of $\bar{\xi}_i$. Now by (\ref{normedeq}),
we have
\begin{align*}
  1 &= g^{-1}(\xi,\xi)\\
  &= \left(\begin{array}{cc}
    \sin(\theta)\xi_0,&\rho\sin(\theta)\xi_{\mu}
  \end{array}\right)B
  \left(\begin{array}{c}
    \sin(\theta)\xi_0\\
    \rho\sin(\theta)\xi_{\mu}\end{array}\right),
\end{align*}
where $B$ is a smooth, uniformly positive definite matrix on $\tU$. Thus, for some $c > 0$, we get
\begin{equation*}
  c(\sin^2(\theta)\xi_0^2 + \delta^{\mu\nu}\rho^2\sin^2(\theta)\xi_{\mu}\xi_{\nu}) \leq 1,
\end{equation*}
from which it follows that $\xi_0 = O(\csc\theta)$ and $\xi_{\mu} = O(\rho^{-1}\csc\theta)$. Hence,
$\bar{\xi}_0 = O(1)$ and $\bar{\xi}_{\mu} = O(\csc\theta) = O(e^t)$, both uniformly in the starting point $q$.
Putting this into (\ref{floweq2}), we find
that $\dot{\bar{\xi}}_{\mu} = O(1)$, from which it follows that we can improve our estimate to $\bar{\xi}_{\mu} = O(t) = O(|\log \sin \theta|)$. Finally, we put this
back into (\ref{normedeq}) and substitute (\ref{rescaled}) to conclude that $\bar{\xi}_0^2 \to 1$ as $t \to \infty$ or $\theta \to 0$, indeed, that $\bar{\xi}_0^2 = 1 + O(e^{-t})
= 1 + O(\sin(\theta))$. Due to the sign of $\dot{\theta}$, we may likewise conclude
that 
\begin{equation*}
  \bar{\xi}_0 = -1 + O(e^{-t}) = -1 + O(\sin(\theta)).
\end{equation*}
We here use that $\rho^2\bar{g}^{00} = 1 + O(\rho\sin(\theta))$.

Because $\dot{\theta} < 0$ for all $t$, we may reparametrize our geodesic equations by $\theta$. This amounts, by the chain rule, to dividing by $\dot{\theta}$,
and by (\ref{dottedeq}), we have 
$\dot{\theta} = \rho\sin^2(\theta)\bar{g}^{0\mu}\bar{\xi}_{\mu} + \rho^2\sin(\theta)\bar{g}^{00}\bar{\xi}_0$. (We recall that
$\bar{g}^{00} = O(\rho^{-2})$, and regard $\rho^2\bar{g}^{00}$ as a single smooth function up to $\rho = 0$, which however does not vanish.) 
Changing variables on the first equation in (\ref{floweq2}) we then get
\begin{equation*}
  \frac{dx^{\mu}}{d\theta} = \frac{\rho\sin^2(\theta)\bar{g}^{\mu\nu}\bar{\xi}_{\nu} + \rho^2\sin(\theta)\bar{g}^{\mu 0}\bar{\xi}_0}
  {\rho\sin^2(\theta)\bar{g}^{0\nu}\bar{\xi}_{\nu} + \rho^2\sin(\theta)\bar{g}^{00}\bar{\xi}_0}
  = \frac{\rho\sin(\theta)\bar{g}^{\mu\nu}\bar{\xi}_{\nu} + \rho^2\bar{g}^{\mu0}\bar{\xi}_0}
  {\rho\sin(\theta)\bar{g}^{0\nu}\bar{\xi}_{\nu} + \rho^2\bar{g}^{00}\bar{\xi}_0}.
\end{equation*}
Now, the denominator is just $\frac{\dot{\theta}}{\sin(\theta)}$, which we know by Proposition \ref{geodprop} is nonzero as a
function of $t$ when $t$ is finite, and thus as a function of $\theta$ when $\theta > 0$. On the other hand, when $\theta \to 0$, the denominator goes to
$-1$ by our above computations of $\bar{\xi}_i$ asymptotics. Thus, the denominator is bounded away from zero, and the equation is smooth in a neighborhood of our
solutions. 

We next study $\frac{\partial \bar{\xi}_{\mu}}{\partial \theta}$. All but two of the terms in $\dot{\bar{\xi}}_{\mu}$ have a factor of $\sin(\theta)$
and thus yield to the same analysis we just performed. Focusing on the remaining terms, we have
\begin{equation*}
  \frac{\partial \bar{\xi}_{\mu}}{\partial \theta} = (smooth) - \frac{\rho_{\mu} + \frac{1}{2}\rho^3\frac{\partial \bar{g}^{00}}{\partial x^{\mu}}\bar{\xi}_0^2}
  {\rho\sin^2(\theta)\bar{g}^{0\mu}\bar{\xi}_{\mu} + \rho^2\sin(\theta)\bar{g}^{00}\bar{\xi}_0}.
\end{equation*}
Now if $\mu = s \neq n$, the numerator vanishes to order $\sin(\theta)$, so again
the equations are smooth in a neighborhood of our solutions. If $\mu = n$, then the numerator is
$1 - \bar{\xi}_0^2 + O(\sin(\theta))$. But $\bar{\xi}_0^2 = 1 + O(\sin(\theta))$, so the numerator is $O(\sin(\theta))$, and this equation is smooth in
a neighborhood of our solutions.
Finally we study $\frac{\partial \bar{\xi}_0}{\partial \theta}$. Once again, only two terms of $\dot{\bar{\xi}}_0$ lack a factor of $\sin(\theta)$, their sum
being $\cos(\theta)(\rho^2\bar{g}^{00}\bar{\xi}_0^2 - 1)$. For the same reasons as before, this is in fact $O(\sin(\theta))$. Thus, the entire
$(x^{\mu},\bar{\xi}_i)$ system is smooth up to $\theta = 0$; and so the solutions are smooth as functions of $\theta$, and depend smoothly on the initial point
$q \in \ctQ{\rho_0}$. All of this analysis is uniform up to $\rho = 0$.

It remains to show that $\theta$ is smooth in $\tau$ up to $\tau = 1$, and depends smoothly on $q$; of course, we already know this for $\tau < 1$.
We have just used $2n + 1$ of the equations in (\ref{floweq2}); the remaining equation, for $\dot{\theta}$, can now be written as a scalar ODE
for $\theta$ in terms of $t$, since $x^{i}$ and $\bar{\xi}_i$ depend smoothly on $\theta$.
Explicitly, the equation is
\begin{equation*}
  \dot{\theta} = \rho\sin^2(\theta)\bar{g}^{0\mu}\bar{\xi}_{\mu} + \rho^2\sin(\theta)\bar{g}^{00}\bar{\xi}_0.
\end{equation*}
The right-hand side is $O(\theta)$, so write the equation as
\begin{equation}
  \label{dotphi}
  \dot{\theta} = -\theta a(\theta),
\end{equation}
Here $a$ is smooth in $\theta$ all the way to $\theta = 0$.
It is clear from our earlier analysis of $\bar{\xi}_0$ that $a(0) = 1$ for all $q$, and also
that $a$ is nonvanishing for $\theta$ along our curves.
We reparametrize $\theta$ by $\tau = 1 - e^{-t}$, and (\ref{dotphi}) becomes
\begin{equation*}
  (\tau - 1)\frac{d\theta}{d\tau} = \theta a(\theta).
\end{equation*}
This is a separable equation; if we write $a(\theta)^{-1} = 1 + \theta b(\theta)$, then the equation has solution
\begin{equation*}
  \theta e^{\int_0^{\theta}b(\zeta)d\zeta} = c(1 - \tau),
\end{equation*}
which holds for $0 \leq \tau \leq 1$, and where $b$ is smooth in both $\theta$ and $q$. Now by the implicit function theorem, this uniquely
defines $\theta$ as a function of $\tau$ and $q$ near $\tau = 1$, smoothly depending on both variables. Thus, as desired, $\theta$ -- and, hence, the entire
$(x^i,\bar{\xi}_i)$ system -- exists and depends smoothly on both $\tau$ and $q$ for $\tau \in [0,1]$ and for $q$ up to $\tS$.

We have still to show that $\exp$ is a local diffeomorphism on $\cbundle{N_+\tQ_{\rho_0}}$. Now it is elementary that, given
a smooth map between manifolds with corner which takes the corner to the corner, the boundary interior to the boundary interior, and the
interior to the interior, it is a local diffeomorphism if and only if its differential is nowhere singular. It suffices, then, to show that
$d\exp$ is everywhere a bijection, or that $\det d\exp \neq 0$. Given this last formulation, it suffices to show this on 
$[0,1] \times (\tQ \cap \tS) 
\subset \cbundle{N_+\tQ_{\rho_0}}$, and it will then follow
for $(\tau,q) \in [0,1] \times \ctQ{\rho_0}$ by shrinking $\rho_0$. 
Now it is clear from (\ref{rho0eq}) that $\exp|_{[0,1] \times (\tQ \cap \tS)}:[0,1] \times (\tQ \cap \tS) \to \tS$
is a diffeomorphism. Moreover, by Proposition \ref{geodprop}(a), $(d\exp)|_{T\cbundle{N_+(\tQ\setminus\tS)}|_{[0,1]\times
(\tQ \cap \tS)}}$ takes nonzero transverse vectors to nonzero transverse vectors. Thus, $d\exp$ is an isomorphism,
and $\exp$ is a local diffeomorphism on all of $\cbundle{N_+\tQ_{\rho_0}}$ (for $\rho_0$ small).

Next we wish to demonstrate (\ref{quadest}) using (\ref{floweq2}) and the equation of variation. We consider perturbations about the solution
(\ref{rho0sol}) starting from $q \in \tQ \cap \tS$, as $q$ varies.
Write (\ref{floweq2}) as
$(x,\bar{\xi})^{\cdot} = F(x,\bar{\xi})$; let $F^{\mu}$ be the component of $F$ corresponding to $x^{\mu}$.
Then the equation of variation tells us that
\begin{equation}
  \label{eqvar}
  \frac{\partial}{\partial t}\frac{\partial x^{\mu}}{\partial \rho(q)} = \frac{\partial F^{\mu}}{\partial x^i}\frac{\partial x^i}{\partial \rho}
  + \frac{\partial F^{\mu}}{\partial \bar{\xi}_i}\frac{\partial \bar{\xi}_i}{\partial \rho};
\end{equation}
and because, at $t = 0$, $x^{\mu}$ are simply the coordinates of $q$, we have
initial condition $\left.\frac{\partial x^{\mu}}{\partial \rho}\right|_{t = 0} = \delta^{\mu n}$. There is additionally an
initial condition for $\left.\frac{\partial \bar{\xi}_i}{\partial\rho}\right|_{t = 0}$, smooth in $q$, but we do not need to write it explicitly.

We claim that the right-hand side of (\ref{eqvar}) is $0$ along our solution. By (\ref{rho0sol}), we have in this case $\rho = 0$ and $\bar{\xi}_{\mu} = 0$, and $\bar{\xi}_0 = -1$.
First consider the first term of (\ref{eqvar}), involving $\frac{\partial F^{\mu}}{\partial x^i}$.
By (\ref{floweq2}), $F^{\mu} = \rho\sin^2(\theta)\bar{g}^{\mu\nu}\bar{\xi}_{\nu} + \rho^2\sin(\theta)\bar{g}^{\mu 0}\bar{\xi}_0$.
Because $\bar{\xi}_{\nu} = 0$ along our solution, the derivative of the first term of $F^{\mu}$ vanishes easily, and the derivative of the second
term vanishes because $\rho = 0$. Very similar considerations show that the second term of (\ref{eqvar}) vanish because $\rho = 0$ along the solution.
Hence, the entire right-hand side of (\ref{eqvar}) vanishes
identically along our solution.  Thus, $\frac{\partial x^{\mu}}{\partial\rho(q)} = \delta^{\mu n} + O(\rho)$, which establishes (\ref{quadest}).

We now turn to the final statement. Let $\tU$, $V$, and $\tQ_{\rho_0}$ be as above.
Notice by (\ref{g}) that $\bar{g}$ extends to $\tS$ as a smooth symmetric positive semidefinite tensor field, and that along
$\tS$, we have $\ker \bar{g} = \spn\left\{ \frac{\partial}{\partial \theta} \right\}$. It follows that $\bar{g}|_{TV}$ is a metric.
Now for any $(\tau,q) \in \cbundle{N_+(V\setminus\tS)}$, we have $T_{(\tau,q)}\cbundle{N_+(V\setminus\tS)} \cong \mathbb{R}\frac{\partial}{\partial \tau} 
\oplus T_qV$ 
canonically. For $0 \leq \tau \leq 1$, define $\exp_{\tau}:\overline{V} \to \tX$ by $\exp_{\tau}(q) = \exp(\tau,q)$. The function
$f:[0,1] \times T\overline{V} \to \tX$ given by
$f(\tau,Y) = |d\exp_{\tau}(Y)|_{\bar{g}}$ is a smooth map. Now
$\exp$ is a local diffeomorphism such that $0 \neq d\exp_{(\tau,q)}\left( \frac{\partial}{\partial \tau} \right) 
\in \spn\frac{\partial}{\partial \theta} = \ker\bar{g}$ for 
$q \in V \cap \tS$. We conclude that $f$ is nonvanishing. Thus, it attains
a positive minimum on the compact set $[0,1] \times S^1_{\bar{g}}T\overline{V}$. This yields the claim.
\end{proof}

The above proof relied in a fundamental way on the behavior of the extension $\exp$ to the boundary $\tS$ in order 
to show that
$\exp$ is a local diffeomorphism. It is possible to give a proof on the interior that $\exp$ is a local diffeomorphism
using Jacobi fields in a more general setting. The following result is unlikely to surprise practitioners in the area,
but we did not find a published proof. Because of its potential applications in other settings, 
it may be worthwhile to record explicitly
in the literature, so we state and prove the result, and then use it in Proposition \ref{difflem} to give an alternate proof
of the local diffeormorphism property in Proposition \ref{extendprop}. The remainder of this section will not be used in
subsequent sections.

\begin{proposition}
  \label{explem}
  Let $\beta > 0$ and $0 < \kappa < \sqrt{\beta}$, and let $(Z,g)$ be a Riemannian manifold with hypersurface 
  $Q$ having unit normal field $\nu$.
  Suppose that $|g^{-1}K| \leq \kappa$ on $Q$, where $K$ is the second fundamental form of $Q$ and $|g^{-1}K|$ is 
  the maximal absolute value of an eigenvalue
  of the shape operator. Moreover, suppose $W \subseteq N_+Q$ is an open
  subset of the one-sided normal bundle to $Q$ having the property that whenever $Y \in W$, $tY \in W$ for $0 \leq t \leq 1$. Finally
  suppose that all sectional curvatures of $g$ are bounded above by $-\beta$ on $\exp(W)$. 
  Then $\exp$ is a local diffeomorphism
  on $W$, and if $\xi:(-\varepsilon,\varepsilon) \to Q$ is a smooth curve with
  $\nu_{\xi(s)} \in W$ and if $\Gamma:[0,a) \times (-\varepsilon,\varepsilon) \to Z$ is given by $\Gamma_t(s) := \Gamma(t,s)
  = \exp(t\nu_{\xi(s)})$, then for all $t \geq 0$ and $s \in (-\varepsilon,\varepsilon)$, 
  $|\Gamma_t'(s)|_g \geq c|\xi'(s)|_g$, where $c = \sqrt{\frac{1}{2}\left( 1 - \frac{\kappa^2}{\beta} \right)}$.
\end{proposition}
\begin{proof}
  Let $\pi:N_+Q \to Q$ be the basepoint map. For convenience, we assume that $\pi(W) = Q$ (or we could just restrict $Q$). For each $p \in Q$,
  we let $\gamma_p$ be the geodesic in $Z$ for which $\gamma(0) = p$ and $\gamma'(0) = \nu_p$.

  Let $(t_0,p) \in [0,\infty) \times Q \approx N_+Q$ be fixed.
  Plainly $d\exp_{(t_0,p)}\left( \frac{d}{dt} \right) = \gamma_p'(t_0) \neq 0$. For $Y \in T_pQ$ with
  $|Y|_g = 1$ for convenience, let $\xi:(-\varepsilon,\varepsilon) \to Q$ be a smooth curve such that
  $\xi(0) = p$ and $\xi'(0) = Y$. Let $a > 0$ be sufficiently small that $a\nu_{\xi(s)} \in W$ for each $s \in (-\varepsilon,\varepsilon)$
  (shrinking $\varepsilon$ if necessary).
  For $(t,s) \in [0,a) \times (-\varepsilon,\varepsilon)$, define $\Gamma(t,s) =
  \gamma_{\xi(s)}(t)$. Then $d\exp_{(t_0,p)}(Y) = \partial_s\Gamma(t_0,0)$. (We are using the identification 
  $T_{(t_0,p)}N_{+}Q
  \cong \mathbb{R} \oplus T_pQ$.)
  This is simply the Jacobi field along $\gamma_p$ defined by the smooth variation $\Gamma$
  evaluated at $t_0$. Thus, since $Y$ is arbitrary, it suffices to show that the Jacobi field $J(t) = \partial_s\Gamma(t,0)$ is nonvanishing
  and is nowhere parallel to $\gamma_p'(t)$. At $t = 0$, we have $J(0) = \xi'(0) = Y \perp \nu_p$. Moreover, by the symmetry lemma
  we have
  \begin{equation*}
    D_tJ(t) = D_t\partial_s\Gamma(t,s)|_{s = 0} = D_s\partial_t\Gamma(t,s)|_{s = 0},
  \end{equation*}
  where $D_t$ denotes covariant differentiation along the curve $t \mapsto \Gamma(t,s)$, and similarly for $D_s$. At $t = 0$, this gives
  \begin{equation*}
    D_tJ(0) = D_s\gamma_{\xi(s)}'(0)|_{s = 0} = D_s\nu \perp \nu = \gamma_p'(0),
  \end{equation*}
  since $\nu$ is a unit vector field. Thus, at $t = 0$, both $J$ and $D_tJ$ are normal to $\gamma_p'$, which implies that
  $J$ is a normal Jacobi field. In particular, if nonvanishing, it is nowhere parallel to $\gamma_p'(t)$.

  Set $f(t) = \langle J(t),J(t)\rangle_g$. Then
  \begin{equation}
    \label{eqfprime}
    f'(t) = \frac{d}{dt}\langle J,J\rangle_g = 2\langle D_tJ,J\rangle_g.
  \end{equation}
  It follows by Cauchy-Schwartz that
  \begin{equation}
    \label{eqcs}
    |f'(t)| \leq 2|f(t)|^{\frac{1}{2}}|D_tJ|_g.
  \end{equation}
  As we have seen, $D_tJ|_{t = 0} = D_s\nu$, which is simply the shape operator applied to $Y = \xi'(0) = J(0)$. It follows by our hypothesis that
$\rho_0$, $\langle D_tJ,J\rangle_g|_{t = 0} \geq -\kappa$, so
  \begin{equation*}
    f'(0) = \frac{d}{dt}\langle J,J\rangle|_{t = 0} \geq -2\kappa.
  \end{equation*}
  Now by (\ref{eqfprime}), (\ref{eqcs}), and the Jacobi equation, whenever $f(t) \neq 0$ we have
  \begin{align}
    f''(t) = \frac{d^2}{dt^2}\langle J,J\rangle_g &= 2\langle D_t^2J,J\rangle + 2\langle D_tJ,D_tJ\rangle\nonumber\\
    &\geq -2R(J(t),\gamma_p'(t),\gamma_p'(t),J(t)) + \frac{1}{2}\frac{f'(t)^2}{f(t)}\nonumber\\
    &> 2\beta f(t) + \frac{1}{2}\frac{f'(t)^2}{f(t)}
  \end{align}
  since the sectional curvature is less than $-\beta$.

  We briefly pause to define weighted hyperbolic trigonometric functions. For $\eta \in \mathbb{R}$, define
  $\cosh_{\eta}(t) = \frac{1}{2}(e^t + \eta e^{-t})$, and $\sinh_{\eta}(t) = \frac{1}{2}(e^t - \eta e^{-t})$. It is easy to show that
  $\cosh_{\eta}'(t) = \sinh_{\eta}(t)$ and $\sinh_{\eta}'(t) = \cosh_{\eta}(t)$, and also that $\sinh_{\eta}^2(t) = \cosh_{\eta}^2(t) - \eta$.

  Now consider the second-order differential equation given by
  \begin{equation*}
    h''(t) = 2\beta h(t) + \frac{1}{2}\frac{h'(t)^2}{h(t)},
  \end{equation*}
  with $h(0) = 1$ and $h'(0) = -2\kappa$. Let $A = \frac{1}{2}\left(1 - \frac{\kappa^2}{\beta}\right) > 0, 
  B = \frac{1}{2}\left(1 - \frac{\kappa}{\sqrt{\beta}}\right)^2 > 0$, and
  $\eta = \frac{(\sqrt{\beta} + \kappa)^2}{(\sqrt{\beta} - \kappa)^2} > 0$. Then a solution to our initial-value problem is $h(t) =
  A + B \cosh_{\eta}(2\sqrt{\beta}t)$, as may be easily checked. We wish to apply Proposition \ref{complem} to show that $f(t) \geq h(t)$, and that
  thus $f$ is bounded below by $A = \frac{1}{2}\left(1 - \frac{\kappa^2}{\beta}\right)$. 
  Define $a:\mathbb{R}^2\setminus (\left\{ 0 \right\} \times \mathbb{R}) \to \mathbb{R}$ by
  \begin{equation*}
    a(u,v) = 2\beta u + \frac{1}{2}\frac{v^2}{u}.
  \end{equation*}
Now define $b:\mathbb{R}^2 \to \mathbb{R}$ by
 \begin{equation*}
    b(u,v) = \left\{
    \begin{array}{lr}
      2\beta u + \frac{1}{2}\frac{v^2}{u} & |v| < 2\sqrt{\beta}|u|\\
      4\beta u & |v| \geq 2\sqrt{\beta}|u|
    \end{array}\right..
  \end{equation*} 
  Note that $b$ is Lipschitz, and that whenever $u > 0$, we have $a \geq b$.

  Now $h$ satisfies $h''(t) - a(h(t),h'(t)) = 0$; and because $|h'(t)| < 2\sqrt{\beta}|h(t)|$, it also satisfies
  $h''(t) - b(h(t),h'(t)) = 0$. Similarly, because $f(0) \geq 0$ and because, whenever $f \neq 0$, $f$ satisfies $f''(t) \geq a(f(t),f'(t))$, it follows
  that, at least up until $f$ vanishes for the first time, we have $f''(t) \geq b(f(t),f'(t))$. But $b$ is Lipschitz and is nondecreasing in $u$.
  Thus, by Proposition \ref{complem}, $f \geq h$ on any interval $I = [0,t_0]$ such that $f$ is nonvanishing on $I$. But since $f$ is continuous, and $h$ is
  bounded away from 0 by $A > 0$, we conclude that $f$ must be everywhere greater than $A$. This yields the claim, taking $c = \sqrt{A}
  = \sqrt{\frac{1}{2}\left( 1 - \frac{\kappa^2}{\beta} \right)}$.
\end{proof}

To apply this to our situation, we require two lemmas that will also be independently useful in applications.

Let $\tU$ be a neighborhood on which a polar identification ($\theta,\pi_S,\rho)$ exists.
\begin{lemma}
  \label{ksmoothlem}
  Let $g$ be an admissible metric on $(\tX,\tM,\tQ,\tS)$. If $\bar{g} = \rho^2\sin^2(\theta)g$ is the compactified metric on the interior of
  $\tU$, then the second fundamental form $\overline{K}$ of $(\tQ \setminus \tS) \cap \tU$ with respect to $\bar{g}$ extends
  smoothly to $\tQ \cap \tU$.
\end{lemma}
\begin{proof}
  Notice once again that by (\ref{g}), $\bar{g}$ extends smoothly to
  $\tU$ as a smooth tensor field (although not as a metric).
  Let $\bar{\nu}$ be the inward unit normal vector field on $\tQ \setminus \tS$ with respect to $\bar{g}$. By Lemma
  \ref{normlem}, $\bar{\nu} = -\frac{1}{\rho}\frac{\partial}{\partial \theta} + O_{\bar{g}}(\rho)$. We next wish to
  consider $\overline{\Gamma}_{i0j} = \frac{1}{2}(\partial_i\bar{g}_{0j} + \partial_{\theta}\bar{g}_{ij} - \partial_j\bar{g}_{i0})$.
  Plainly this is smooth on $\tU$. But moreover, by (\ref{g}), we see that it is $O(\rho)$. (Remember that
  $k_{\rho}$ is independent of $\theta$). All other Christoffel symbols are likewise smooth.
  Now using Weingarten's equation, we have in coordinates
  \begin{align*}
    \overline{K}_{ij} &= -\bar{g}_{kj}\overline{\nabla}_i\bar{\nu}^k\\
    &= -\bar{g}_{kj}\partial_i\bar{\nu}^k - \bar{\nu}^l\overline{\Gamma}_{ilj}\\
    &= \bar{g}_{0j}\partial_i(\rho^{-1}) + (smooth).
  \end{align*}
  Since $\bar{g}_{0j} = O(\rho^2)$ for any $0 \leq j \leq n$, we conclude that $\overline{K}$ extends smoothly
  to $\tQ \cap \tS$.
\end{proof}

\begin{lemma}
  \label{curvlem}
  Let $g$ be an admissible metric on $(\tX,\tM,\tQ,\tS)$ and $R$ its curvature tensor. Then
  $R_{ijkl} + (g_{ik}g_{jl} - g_{il}g_{jk}) = O_g(\rho\sin\theta)$.
\end{lemma}
\begin{proof}
  We begin by showing that $T_{ijkl} := R_{ijkl} + (g_{ik}g_{jl} - g_{il}g_{jk}) = O_g(\rho)$,
  using a modification of the proof of Proposition 1.10 of \cite{maz86}.
 
  Let $r = \rho\sin\theta$, so that $\bar{g} = r^2g$.
  Now the standard formula for conformal change of the Riemann tensor shows that
  \begin{equation*}
    R_{ijkl} = r^{-2}\overline{R}_{ijkl} + r^{-3}\left( r_{jk}\bar{g}_{il} + r_{il}\bar{g}_{jk} -
    r_{ik}\bar{g}_{jl} - r_{jl}\bar{g}_{ik}\right) - |\nabla r|_{\bar{g}}^2\left( g_{il}g_{jk} -
    g_{ik}g_{jl}\right),
  \end{equation*}
  where $r_{jk}$ represents the Hessian of $r$ taken with respect to $\bar{g}$.
  Thus, the expression we are interested in takes the form
  \begin{multline}
    \label{Rinterest}
    R_{ijkl} + (g_{ik}g_{jl} - g_{il}g_{jk}) = r^{-2}\overline{R}_{ijkl} + r^{-3}(r_{jk}\bar{g}_{il} + r_{il}\bar{g}_{jk}
    - r_{ik}\bar{g}_{jl} - r_{jl}\bar{g}_{ik})\\
    + (g_{il}g_{jk} - g_{ik}g_{jl})(1 - |\nabla r|_{\bar{g}}^2).
  \end{multline}
  Now using the fact that $r = \rho \sin \theta$ and (\ref{ginv}), it follows immediately that
  $|\nabla r|_{\bar{g}}^2 = 1 + O(\rho)$. Thus, the last term is $O_g(\rho)$.
  
   We next turn to computing $\overline{R}_{ijkl}$. It will be convenient to use the formula
  \begin{equation*}
    \overline{R}_{ijkl} = \frac{1}{2}\left( \partial_{jl}^2\bar{g}_{ik} + \partial_{ik}^2\bar{g}_{jl} - 
    \partial_{il}^2\bar{g}_{jk}
    -\partial_{jk}^2\bar{g}_{il}\right) + \bar{g}^{pq}\left( \overline{\Gamma}_{jlp}\overline{\Gamma}_{ikq} - \overline{\Gamma}_{jkp}
    \overline{\Gamma}_{ilq}\right),
  \end{equation*}
  where $\overline{\Gamma}_{ijk} = \frac{1}{2}(\partial_i\bar{g}_{jk} + \partial_j\bar{g}_{ij} - \partial_k\bar{g}_{ij})$. 
  We thus compute
  these Christoffel symbols. Using our polar $g$ coordinates and (\ref{g}), we find that
  \begin{equation*}
    \begin{array}{lll}
      \overline{\Gamma}_{000} = O(\rho^3) & \overline{\Gamma}_{00u} = O(\rho^3) & \overline{\Gamma}_{00n} = -\rho + O(\rho^2)\\
      \overline{\Gamma}_{0s0} = O(\rho^2) & \overline{\Gamma}_{0su} = O(\rho) & \overline{\Gamma}_{0sn} = O(\rho)\\
      \overline{\Gamma}_{0n0} = \rho + O(\rho^2) & \overline{\Gamma}_{0nu} = O(\rho) & \overline{\Gamma}_{0nn} = O(\rho)\\
      \overline{\Gamma}_{st0} = O(\rho) & \overline{\Gamma}_{stu} = O(1) & \overline{\Gamma}_{stn} = O(1)\\
      \overline{\Gamma}_{sn0} = O(\rho) & \overline{\Gamma}_{snu} = O(1) & \overline{\Gamma}_{snn} = O(\rho)\\
      \overline{\Gamma}_{nn0} = O(\rho) & \overline{\Gamma}_{nnu} = O(1) & \overline{\Gamma}_{nnn} = O(1).
    \end{array}
  \end{equation*}
  Now using these computations, (\ref{g}), and (\ref{ginv}), it follows straightforwardly that
  \begin{align*}
    \overline{R}_{0\mu\nu0} &= O(\rho)\\
    \overline{R}_{i\mu\nu\sigma} &= O(1),
  \end{align*}
  where $1 \leq \mu \leq n$ and $0 \leq i \leq n$. It follows, since $\sin(\theta)\frac{\partial}{\partial \theta}$
  and $\rho\sin(\theta)\frac{\partial}{\partial x^{\mu}}$ are a basis of approximately $g$-unit vector fields, that
  $r^{-2}\overline{R} = O_g(\rho)$.

  Finally, we compute that
  \begin{equation*}
    r_{ij} = 2\rho_{(i}\theta_{j)}\cos(\theta) - \sin(\theta)\bar{g}^{nk}\overline{\Gamma}_{ijk} -
    \rho\cos(\theta)\bar{g}^{0k}\overline{\Gamma}_{ijk},
  \end{equation*}
  from which it follows that
  $r_{\mu\nu} = O(1)$, that $r_{0\mu} = O(\rho)$, and that $r_{00} = O(\rho^2)$.
  Thus, the second term of (\ref{Rinterest}) is also $O_g(\rho)$, which yields the claim that $T = O_g(\rho)$.

  Now, near neighborhoods in $\tM$ away from $\tM \cap \tS$, it follows from the usual curvature result
  for asymptotically hyperbolic spaces, given in Proposition 1.10 of \cite{maz86}, that
  $T = O_g(\sin(\theta))$. Thus, the result will follow if we can show that $T$ is smooth as a section
  of the bundle $\otimes^4({}^{0e}T^*\tX)$. But this follows from what we have already done, and in
  particular from (\ref{Rinterest}). First, $r^{-2}\overline{R}_{ijkl}$ is smooth as a section of the bundle,
  since no more than two of the indices can be $0$, and a smooth frame for ${}^{0e}T^*\tX$
  is given by (\ref{dualframe}).
  But the second term similarly is smooth, for we have just seen that $r_{jk} = O(\rho)$ whenever
  either index is $0$. Thus, $T$ is a smooth section and is $O_g(\rho)$ and $O_g(\sin(\theta))$. The result follows.
\end{proof}

We can now give the alternate proof of the local diffeomorphism property.
\begin{proposition}
  \label{difflem}
  Let $(\tX,\tM,\tQ,\tS)$ be the blowup of the cornered space $(X,M,Q)$, and let $g$ be an admissible metric on $\tX$.
  There exists $\rho_0 > 0$ such that the map $\exp:N_+\tQ_{\rho_0} \to \tX$ is a local diffeomorphism on the normal bundle $N_+\tQ_{\rho_0}$.
  Moreover, there exists some $c > 0$ such that, if $\xi:(-\varepsilon,\varepsilon) \to \tQ_{\rho_0}$ is a smooth curve and
  $\Gamma:[0,\infty)\times(-\varepsilon,\varepsilon) \to \tX$ is given by $\Gamma_t(s) := \Gamma(t,s)
  = \exp(t\nu_{\xi(s)})$, then for all $t \geq 0$ and $s \in (-\varepsilon,\varepsilon)$,
  we have $|\Gamma_t'(s)|_g \geq c|\xi'(s)|_g$.
\end{proposition}
\begin{proof}
  It suffices to prove the second claim. Let $0 < \kappa < 1$ be such that $|\cos\theta| < \kappa$ on $\tQ \cap \tS$, which exists
  by compactness. Also let $\kappa < \beta < 1$.

  By Lemma \ref{curvlem}, there is some $\rho_{\beta}$ such that the sectional curvatures of $g$ are strictly less than
  $-\beta$ for all $x \in \mathring{\tX}_{\rho_{\beta}}$. By Proposition \ref{geodprop}, we can choose $\rho_0 > 0$ such that, for $q \in \tQ_{\rho_0}$,
  $\gamma_q$ remains in $\tX_{\rho_{\beta}}$.

  We begin by studying the eigenvalues of the second fundamental form of $\tQ$, which we denote by $K(Y,Z) =
  \langle\nabla_YZ,\nu\rangle_g$ (and correspondingly for $\overline{K}$ with respect to $\bar{g}$), and to do this we
  first compute the compactified second fundamental form $\overline{K}$. For $r \neq 0$, the unit $\bar{g}$-normal
  vector field to $\tQ_{\rho_0}$ is given by $\bar{\nu} = r^{-1}\nu$. (Recall that $r = \rho\sin\theta$.)
  A straightforward computation shows that for any vector fields $X, Y$ tangent to $\tQ$, we have
  \begin{equation*}
    \overline{\nabla}_XY = \nabla_XY + r^{-1}\left[ dr(X)Y + dr(Y)X - \langle X,Y\rangle_{\bar{g}}\grad_{\bar{g}}r \right].
  \end{equation*}
  For $q \in \tQ_{\rho_0}$ and $X, Y \in T_q\tQ$, it follows (taking extensions where necessary) that
  \begin{align*}
    \overline{K}(X,Y) &= -\langle\overline{\nabla}_X(r^{-1}\nu),Y\rangle_{\bar{g}}\\
    &= -r^{-1}\langle\nabla_X\nu + dr(X)\bar{\nu} + dr(\bar{\nu})X - \langle X,\bar{\nu}\rangle\grad_{\bar{g}}r - dr(X)\bar{\nu},Y\rangle_{\bar{g}}\\
    &= -r^{-1}(r^2K(X,Y) - \bar{g}(X,Y)dr(\bar{\nu})).
  \end{align*}
  Now let $Y, Z \in TQ$ be $g$-unit vectors over the same point, and let $\overline{Y} = r^{-1}Y$ and
  $\overline{Z} = r^{-1}Z$ be the parallel $\bar{g}$-unit vectors. It follows that
  \begin{equation*}
    K(Y,Z) = g(Y,Z)dr(\bar{\nu}) + r\overline{K}(\overline{Y},\overline{Z}).
  \end{equation*}
  Now $dr = \sin\theta d\rho + \rho\cos\theta d\theta$, and by Lemma \ref{normlem}, $\bar{\nu} = (-\frac{1}{\rho} + O(1)) \frac{\partial}{\partial \theta}
  + O(\rho)$. Thus, $|dr(\bar{\nu})| \to |\cos(\theta)| < \kappa$ as $\rho \to 0$. Since $\overline{K}$ is smooth on all
  of $\ctQ{\rho_0}$ by Lemma \ref{ksmoothlem}, 
  it follows that for $\rho$ small enough, the eigenvalues of the shape operator $g^{-1}K$ are 
  bounded in absolute value by $\kappa$:
  $|\lambda| < \kappa$. We restrict $\rho_0$ if necessary to ensure this condition.

  The result now follows straightforwardly by applying Proposition \ref{explem} with $Z = X$, with $Q = \tQ_{\rho_0}$, and with
  $W = N_+\tQ_{\rho_0}$.
\end{proof}

\section{Injectivity}\label{injectivity}

In the preceding sections, we have shown that there is a neighborhood $\ctQ{\rho_0}$ of $\tS$ in $\tX$ such that
$\exp:\cbundle{N_+\tQ_{\rho_0}} \to \tX$ is a local diffeomorphism. The remaining step to show that $\exp$ is a diffeomorphism onto its image
is to prove injectivity.

We will first work on the interior or non-compactified normal bundle $N_+\tQ$ near a fixed point of $\tQ \cap \tS$.
We will then make the result global along $\tQ \cap \tS$ using a compactness argument.

We prove injectivity on the interior using a homotopy lifting argument whose structure is that of Theorem 2 in \cite{her63}. We first prove a lifting result.
We let $\pi:N_+(\tQ\setminus \tS) \to (\tQ\setminus\tS)$ be the basepoint map.

\begin{proposition}
  \label{lift}
  Let $c, \rho_0$ be as in Proposition \ref{extendprop}. 
  Let $W \subset \tQ_{\frac{\rho_0}{2}}$ be open, $q \in W$, and let
  $x = \exp t_x \nu_q$ for some $t_x > 0$.
  Let $\alpha:[0,l] \to \mathring{\tX}$ be a smooth curve such that
  $\alpha(0) = x$ and such that
  \begin{equation}
    \label{nointercond}
    \alpha([0,l]) \cap \exp(\pi^{-1}(\partial W)) = \emptyset.
  \end{equation}
  Then there is a unique smooth curve $\sigma:[0,l] \to N_+W$ such that $\sigma(0) = t_x\nu_q$ and
  $\exp \sigma(s) = \alpha(s)$.
  Let $\xi = \pi\circ\sigma:[0,l] \to W$. Then $L_g(\xi) \leq c^{-1}L_g(\alpha)$.

  Moreover, if $\alpha:[0,1] \times [0,l] \to \mathring{\tX}$ is a homotopy of smooth curves such that,
  for each $\tau$, $\alpha(\tau,0) = x$ and the curve $s \mapsto \alpha(\tau,s)$ satisfies (\ref{nointercond}),
  then there is a unique lift of $\alpha$ to a homotopy of curves based at $t_x\nu_q$. That is,
  there is a unique smooth map $\sigma:[0,1] \times [0,l] \to N_+W$ such that $\sigma(\tau,0) = t_x\nu_q$
  for each $\tau$ 
  and such that $\exp\circ\sigma = \alpha$.
\end{proposition}
\begin{remark}
  We are especially interested in the special case where $W = \tQ_{\frac{\rho_0}{2}}$ itself.
\end{remark}

\begin{proof}
  Let $x, \alpha$ be as in the statement. 
  By Proposition \ref{extendprop}, $\exp:N_+\tQ_{\frac{\rho_0}{2}} \to X$ is a local diffeomorphism. 
  Hence, at least some opening interval
  of $\alpha$ may be lifted uniquely to a smooth curve $\sigma$ beginning at $t_x\nu_q$ -- that is,
  there is some $a > 0$ and a unique smooth $\sigma:[0,a] \to N_+\tQ_{\frac{\rho_0}{2}}$ such that $\sigma(0) = t_x\nu_q$ and
  $\exp\circ\sigma = \alpha|_{[0,a]}$.
  Suppose we cannot lift the entire curve, and let
  $b$ be the supremum of $a > 0$ such that we can uniquely lift $\alpha|_{[0,a]}$ in the preceding
  sense. Then there is a unique lift $\sigma$ of $\alpha|_{[0,b)}$. By continuity and (\ref{nointercond}),
  $\sigma$ takes values in $N_+(\overline{W}\setminus\tS)$.

  As in the statement, define $\xi:[0,b) \to W$ by
  $\xi = \pi\circ\sigma$. By the canonical identification $N_+W\approx [0,\infty) \times W $, we can
  write $\sigma(s) = (t(s),\xi(s))$. Now for each $s < b$, $\exp \sigma(s) = \alpha(s)$; it follows that, for $s < b$, 
  $(d\exp)(\sigma'(s)) = \alpha'(s)$; and thus that $|(d\exp)(\sigma'(s))|^2_g = |\alpha'(s)|^2_g$. Under
  the identification $T_{(t,q)}N_+W \approx \mathbb{R}\frac{\partial}{\partial t} 
  \oplus T_qW$, we can write $\sigma'(s) = \dot{t}(s)\frac{\partial}{\partial t} + \xi'(s)$.
  Let $A = \sup_{0 \leq s \leq b}|\alpha'(s)|_g^2$. Then since
  $(d\exp)\left(\frac{\partial}{\partial t}\right) \perp (d\exp)(\xi'(s))$,
  we have
  \begin{align*}
    A &\geq |\alpha'(s)|_g^2 = \left|\dot{t}(s)(d\exp)\left( \frac{\partial}{\partial t} \right) + 
    (d\exp)_{t\nu_{\xi}}(\xi'(s))\right|_g^2\\
    &= \dot{t}(s)^2 + |(d\exp)_{t\nu_{\xi}}(\xi'(s))|_g^2\\
    &\geq \dot{t}(s)^2 + c^2|\xi'(s)|_g^2 \text{ (by Proposition \ref{difflem})}\numberthis\label{liftbound}.
  \end{align*}
  Thus, both $|\dot{t}(s)|$ and $|\xi'(s)|_g$ are bounded. It follows that
  $\lim_{s \to b}\xi(s)$ exists in $\overline{W}$, so $\xi$ may be extended to exist continuously on
  $[0,b]$ (although \emph{a priori} $\xi(b)$ may not lie in $W$). Because $\xi:[0,b] \to \overline{W}$ has
  finite length by (\ref{liftbound}), $\xi(b) \notin \tS$.
  
  Now also by (\ref{liftbound}), $\lim_{s \to b}t(s)$ exists;
  so $\lim_{s \to b}\sigma(s)$ exists, and $\sigma$ may be continuously extended to $[0,b]$, possibly taking
  values in $N_+\overline{W} \supset N_+W$. However, by continuity we have $\exp\sigma(b) = \alpha(b)$. Because
  (\ref{nointercond}) holds, $\xi(s) \notin \partial W$ for any $0 \leq s \leq b$. Therefore,
  $\xi([0,b]) \subset W$, and hence $\sigma([0,b]) \subset N_+W$. Now by Proposition \ref{difflem} and
  because $W \subseteq \tQ_{\frac{\rho_0}{2}}$, $\exp$ is a local
  diffeomorphism on some ball about $\sigma(b)$, so it follows that $\sigma$ can be smoothly and uniquely extended
  at least some distance beyond $b$. This is a contradiction, so $\sigma$ can be extended smoothly and uniquely
  to all of $[0,l]$.

  We now turn to homotopy lifting. Suppose that $x$ is as above, and that $\alpha:[0,1] \times [0,l] \to 
  \mathring{\tX}$ is a smooth
  map such that $\alpha(\tau,0) = x$ for all $\tau$ and such that, for fixed $\tau$, the curve $s \mapsto \alpha(\tau,s) =:
  \alpha_{\tau}(s)$ meets condition (\ref{nointercond}). 
  We wish to show that there is a lift $\sigma:[0,1] \times [0,l] \to N_+W$
  such that $\exp \sigma = \alpha$. In the following, we will also use the notation $\alpha^s(\tau) = \alpha(\tau,s)$.

  Let $\sigma_0:[0,l] \to N_+W$ be a lift, as above, of $\alpha_0$ beginning at
  $t_x\nu_q$. For each $s \in [0,l]$, let
  $\tau \mapsto \sigma(\tau,s)$ be the lift, starting at $\sigma_0(s)$, of the map
  $\tau \mapsto \alpha(\tau,s)$.
  Then $\sigma:[0,1] \times [0,l] \to N_+W$ and $\alpha(\tau,s) = \exp\circ\sigma(\tau,s)$.
  It is plain that $\sigma$ is smooth in $\tau$. We wish to show that it is smooth in $s$. 
  
  Let $s_0 \in (0,l]$. We will construct a small strip in
  $[0,1] \times [0,l]$, containing $[0,1] \times \left\{ s_0 \right\}$, such that
  $\sigma$ is smooth on the strip. (The case $s_0 = 0$ is easy because $\exp$ is a local diffeomorphism and
  $\alpha_\tau(0) = x$ for all $\tau$). Let $U \subset N_+W$ be a coordinate ball, containing $\sigma(0,s_0)$, such
  that $\exp|_{U}$ is a diffeomorphism. Then $\alpha^{-1}(\exp(U)) \subseteq [0,1] \times [0,l]$ is an open
  neighborhood of $(0,s_0)$. By continuity of $\sigma_0$, there is some $\varepsilon > 0$ such that
  $\sigma_0([s_0 - \varepsilon,s_0 + \varepsilon]) \subset U$. Now if $s \in [s_0 - \varepsilon,s_0 + \varepsilon]$
  and $a > 0$ is small enough that $[0,a] \times \left\{ s \right\} \subset \alpha^{-1}(\exp(U))$,
  then for $0 \leq \tau \leq a$, the map $\tau \mapsto (\exp|_U)^{-1}(\alpha(\tau,s))$ is a smooth lift
  of $\tau \mapsto \alpha(\tau,s)$ beginning at $\sigma_0(s)$. It follows by uniqueness of lifting that
  it is equal to $\tau \mapsto \sigma(\tau,s)$. Thus, on some neighborhood of $\left\{ 0 \right\} \times
  [s_0 - \varepsilon,s_0 + \varepsilon]$, we have $\sigma = (\exp|_U)^{-1}\circ\alpha$, so in particular,
  $\sigma$ is smooth on a neighborhood of $(0,s_0)$.
  
  Set
  \begin{equation*}
    b = \sup \left\{ d \geq 0: \sigma \text{ is smooth on a neighborhood of } [0,d] \times \left\{ s_0 \right\}.\right\}.
  \end{equation*}
  The preceding discussion shows that $b > 0$. Clearly $b \leq 1$,
  We claim $b = 1$. Suppose not, by way of contradiction. Once more, let $U$ be an open set
  containing $\sigma(b,s_0)$ such that $\exp|_U$ is a diffeomorphism. Then again,
  $\alpha^{-1}(\exp(U)) \subseteq [0,1] \times [0,l]$ is an open neighborhood of $(b,s_0)$. Moreover,
  because $\sigma^{s_0}$ is smooth in $\tau$, we have $\sigma(a,s_0) \in U$ for all $a$ sufficiently near $b$.
  Let $a_1 < b$ be sufficiently near. Then $\sigma$ is smooth on some neighborhood of $[0,a_1] \times \left\{ s_0 \right\}$
  by definition of $b$. Hence by continuity, we conclude that $\sigma$ maps some neighborhood of $(a_1,s_0)$ into
  $U$. We can choose some $\varepsilon > 0$ such that $\sigma(\left\{ a_1 \right\} \times [s_0 - \varepsilon,
  s_0 + \varepsilon]) \subset U$ and so that $\sigma$ is smooth on a neighborhood of
  $[0,a_1] \times [s_0 - \varepsilon,s_0 + \varepsilon]$. Thus, for $s \in [s_0 - \varepsilon,s_0 + \varepsilon]$,
  $\sigma(a_1,s) = (\exp|_U)^{-1}\circ\alpha(a_1,s)$. Now, by shrinking $\varepsilon$ if need be,
  we can choose $a_2 > b$ such that $[a_1,a_2] \times [s_0 - \varepsilon,s_0 + \varepsilon] \subseteq
  \alpha^{-1}(\exp(U))$. Fix $s \in [s_0 - \varepsilon,s_0 + \varepsilon]$. The map
  $\tau \mapsto (\exp|_U)^{-1}(\alpha(\tau,s))$ (where $a_1 \leq \tau \leq a_2)$ is a smooth lift of the map
  $\tau \mapsto \alpha(\tau,s)$ beginning at $\sigma(a_1,s)$. Then by uniqueness of path lifts, we have
  $\sigma(\tau,s) = (\exp|_U)^{-1}(\alpha(\tau,s))$ on this rectangle. Thus, $\sigma$ is smooth
  on $[0,a_2] \times [s_0 - \varepsilon,s_0 + \varepsilon]$, which is a contradiction since $a_2 > b$. Thus $b = 1$.
  We conclude that $\sigma$ is smooth on all of $[0,1] \times [0,l]$.\
\end{proof}
\begin{remark}
  It follows from the proof that, if we set $W = \tQ_{\frac{\rho_0}{2}}$ and $R = \left\{ q \in \tQ:
  \rho(q) = \frac{\rho_0}{2}\right\}$, then the hypothesis (\ref{nointercond}) could be
  replaced by the condition $d_{\tQ}(q,R) > c^{-1}L_g(\alpha)$.
\end{remark}

This result in hand, we may prove interior injectivity.

\begin{proposition}
  \label{injprop}
  There exists $a > 0$ such that $\exp:N_+\tQ_a \to \tX$ is injective.
\end{proposition}
\begin{proof}
  Let $\rho_0$ be small enough
  that Propositions \ref{extendprop} and \ref{lift} hold on $\tQ_{\rho_0}$, then define $R \subset \tQ$ by
  $R = \left\{ q \in \tQ: \rho(q) = \frac{\rho_0}{2}\right\}$.   

  We first prove a result with a topological hypothesis.
  By Proposition \ref{geodprop}, there exists $\rho_1 < \frac{\rho_0}{2}$ such that
  if $q \in \tQ_{\rho_1}$, then $\gamma_q$ will not intersect $\exp(\pi^{-1}(R))$:
  for if $q^{\prime} \in R$, then by Proposition \ref{geodprop}(a), $\frac{\varepsilon \rho_0}{2}
  <  \rho(\gamma_{q'}(t))$ for all $t$, whereas  $\gamma_q$ can be made to remain arbitrarily close to $\tS$ by choosing
  $\rho_1$ sufficiently small. 

  We will show that if $V \subseteq \tQ_{\rho_1}$ is connected, and $\tA \subset \tX \setminus (\tM \cup \tS)$ is a 
  simply connected open set such that $\tA \cap \exp(\pi^{-1}(R)) = \emptyset$ and such
  that $\exp(N_+V) \subseteq \tA$, then $\exp$ is injective on $N_+V$. Here, once again, $\pi:N_+(\tQ \setminus \tS) \to \tQ \setminus \tS$
  is the basepoint map.

  Suppose, by way of contradiction, that $\exp$ is not injective on $N_+V$. Then there exists some $x \in \tA$,
  $u \neq v \in N_+V$ such that $\exp(u) = x = \exp(v)$. Let $\sigma:[0,1] \to N_+V$ be a smooth path in $N_+V$
  from $u$ to $v$, and let $\alpha = \exp\circ\sigma:[0,1] \to \tA$. Then $\alpha$ is a smooth loop
  segment at $x$; and since $\tA$ is simply connected, there is a smooth homotopy $\tilde{\alpha}:[0,1] \times
  [0,1] \to \tA$ such that $\tilde{\alpha}(0,s) = \alpha(s)$ and $\tilde{\alpha}(1,s) = x$. Now by construction,
  $\tilde{\alpha}(\tau,s)$ avoids $\exp(\pi^{-1}(R))$ for all $\tau, s$; so by Proposition \ref{lift} with $W = 
  \tQ_{\frac{\rho_0}{2}}$, there exists
  a lifted homotopy $\tilde{\sigma}:[0,1] \times [0,1] \to N_+\tQ_{\frac{\rho_0}{2}}$, 
  based at $u$, such that $\exp\circ\tilde{\sigma}
  = \tilde{\alpha}$. Thus, $\exp\circ\tilde{\sigma}(\tau,1) = x$ for all $\tau$. Moreover, $\tilde{\sigma}(0,1) =
  v$ and $\tilde{\sigma}(1,1) = u$ (since the lift of a constant path is constant). Thus, defining
  $\zeta:[0,1] \to N_+\tQ_{\frac{\rho_0}{2}}$ by $\zeta(\tau) = \tilde{\sigma}(\tau,1)$, the curve $\zeta$ must be a smooth path
  from $v$ to $u$. On the other hand, $\exp\circ\zeta(\tau) = x$ for all $\tau$. But as $\exp$ is a local
  diffeomorphism, $\exp^{-1}(\left\{ x \right\})$ is discrete. Thus, $\zeta$ is a non-constant smooth
  map from a connected space to a discrete space, which is a contradiction. Hence, $\exp$ is injective on $N_+V$, which establishes our claim.

  To allow a general topology, and in particular a full neighborhood of $\tS$, 
  first note that if $B \subseteq S$ is simply connected, then so is $b^{-1}(B)$. This is because $\tS \to S$ is a trivial fibration,
  since $S$ is the intersection of two globally defined hypersurfaces.
  For such $B$, and for $\kappa > 0$, we define
  \begin{equation*}
    \tA(B,\kappa) = \left\{ (\theta,p,\rho) \in \tX: p \in B \text{ and } 0 < \rho < \kappa \right\},
  \end{equation*}
  where we are using our polar identification. Notice that by taking $\kappa$ small enough, we may always assure that
  $\exp(\pi^{-1}(R)) \cap \tA(B,\kappa) = \emptyset$. Also,
  $\tA(B,\kappa)$ will be simply connected for $\kappa$ small.
  Now let $\delta > 0$ be less than the injectivity radius of $S$ with respect to $k_0 = \bar{g}|_{TS}$. For
  $p \in S$, let $B_{\delta}(p)$ denote the $\delta$-ball about $p$ with respect to $k_0$. Thus, for each $p \in S$,
  $B_{\delta}(p) \subseteq S$ is simply connected. Let $\kappa_0 > 0$ be small enough for the
  above conditions to hold for $B_{\delta}(p)$ at every $p \in S$. Set $\tA_p = \tA(B_{\delta}(p),\kappa_0)$.

  By Proposition \ref{geodprop}, there are $\varepsilon > 0, \kappa_1 > 0$ such that for each $p \in S$,
  $\exp(N_+(\tA(B_{\varepsilon}(p),\kappa_1) \cap \tQ)) \subseteq \tA_p$.
  Set $V_p = \tA(B_{\frac{\varepsilon}{4}}(p),\kappa_1) \cap \tQ$. Then $\left\{\underline{V_p} \right\}_{p \in S}$ covers 
  $\tS \cap \tQ \approx S$, and by compactness
  we may take a finite subcover $\left\{ \underline{V_{p_i}} \right\}_{i = 1}^N$. We will denote $V_i = V_{p_i}$.
  Now since there are finitely many $V_i$, we may apply Proposition \ref{geodprop}(b) to conclude that by shrinking $\kappa_1$ if necessary,
  we may ensure that $\exp(N_+V_i) \cap \exp(N_+V_j) = \emptyset$ whenever $\overline{V_i} \cap \overline{V_j} = \emptyset$.

  Set $V = \cup_{i}V_i$. We claim that $\exp$ is injective on $N_+V$. For suppose that there exist
  $u_1 = t_1\nu_{q_1}, u_2 = t_2\nu_{q_2} \in N_+V$
  such that $\exp(u_1) = \exp(u_2)$. By what has just been said, we must have
  $q_1 \in V_i$ and $q_2 \in V_j$ where $\overline{V_i} \cap \overline{V_j} \neq \emptyset$. 
  Let $p_i = \pi_S(q_i)$.
  Then by definition of $V_i, V_j$, we have $q_1, q_2 \in V_{\varepsilon} := \tA(B_{\varepsilon}(p_1),\kappa_1) \cap \tQ$; and
  since, by choice of $\varepsilon$, we have $\exp(V_{\varepsilon}) \subseteq \tA_{p_1}$, and by the simply connected case,
  we may conclude that $q_1 = q_2$. Thus taking $a$ small enough that $\tQ_a \subseteq V$ yields the theorem.
\end{proof}

This result may be extended to the compactified bundle, $\cbundle{N_+\tQ_{a}}$. 

\begin{proposition}
  \label{injprop2}
  There exists $a > 0$ such that $\exp:\cbundle{N_+\tQ_{a}} \to \tX$ is injective.
\end{proposition}
\begin{proof}
  Let $a$ be as in Proposition \ref{injprop}. 
  As before, we label points in $\cbundle{N_+\tQ_{a}}$ by $(\tau,q) \in [0,1] \times \ctQ{a}
  \approx \cbundle{N_+\tQ_{a}}$. Throughout this proof, we regard $\exp$ as a function of
  $\tau \in [0,1]$ and $q \in \tQ_a$.
  We already know that $\exp$ is injective when restricted to the set $\left\{ 0 \leq \tau < 1, \rho(q) > 0 \right\}$. We may quickly
  extend this to $[0,1] \times S$: the exponential map takes $[0,1] \times S$ injectively to $\tS$ by the explicit solution (\ref{ssoln}), 
  and by Proposition \ref{geodprop} no other points are mapped to $\tS$.
  Moreover, since $\left\{ 1 \right\} \times \tQ_a$ is mapped to $\tM \setminus \tS$, its image is disjoint from the image of
  $([0,1) \times \tQ_a) \cup ([0,1] \times (\tS \cap \tQ))$. Therefore, we need only show that for $q_1 \neq q_2 \in \tQ_a$,
  $\exp(1,q_1) \neq \exp(1,q_2)$.

  Suppose, by way of contradiction, that $\exp(1,q_1) = \exp(1,q_2)$, with $q_1,q_2 \in \tQ_a$. Let $\tQ_a \supseteq 
  B_1 \ni q_1$ be open such that $q_2 \notin B_1$.
  Then $[0,1] \times B_1$ is open in $\cbundle{N_+\tQ_{a}}$, and so since $\exp$ is a local diffeomorphism,
  $\exp([0,1] \times B_1)$ is open in $\tX$. Let $\hat{\gamma}_{q_2}:[0,1] \to \tX$ be the rescaled geodesic given by
  $\hat{\gamma}_{q_2}(\tau) = \exp(\tau,q_2)$, which in particular is continuous. Thus, $\hat{\gamma}_{q_2}^{-1}(\exp([0,1] \times B_1))$
  is nonempty and is open in $[0,1]$, and so there is some $\tau_2 < 1$ such that $\hat{\gamma}_{q_2}(\tau_2) \in \exp([0,1] \times B_1)$. Since
  $\tau_2 < 1$, we must have $\hat{\gamma}_{q_2}(\tau) \in \mathring{\tX}$, and so there is some $q_3 \in B_1$, $\tau_3 \in [0,1)$ such that
  $\exp(\tau_3,q_3) = \exp(\tau_2,q_2)$. This contradicts interior injectivity and thus Proposition \ref{injprop}.
\end{proof}

\section{Proofs of Theorems}\label{proofs}

\begin{proof}[Proof of Theorem \ref{mainthm}]
  By Proposition \ref{extendprop}, there exists $\rho_0 > 0$ such that the exponential map
  $\exp:N_+\tQ_{\rho_0}\to\tX$ extends to a local diffeomorphism $\exp:\cbundle{N_+\tQ_{\rho_0}} \to \tX$.
  By Proposition \ref{injprop2}, it is injective for $\rho_0$ small enough.
  Taking $V = \ctQ{\rho_0}$ and $\tU = \exp(V)$, this yields the claim.
\end{proof}

\begin{proof}[Proof of Theorems \ref{normform} and \ref{normformconst}]
  By Theorem \ref{mainthm}, we may take $\rho_0 > 0$ such that $\exp:\cbundle{N_+\tQ_{\rho_0}} \to \tX$ is a diffeomorphism
  onto its image. Let $V \subset \ctQ{\rho_0}$ be a neighborhood of $\tQ \cap \tS$, and set $\tU = \exp(\cbundle{N_+(V \setminus \tS)})$. 
  Now under the canonical decomposition
  $N_+(V \setminus \tS) \approx \mathbb{R}_t \times (V \setminus \tS)$, we have
  \begin{equation}
    \label{geodesicform}
    (\exp|_{N_+(V \setminus \tS)})^*g = dt^2 + g_t,
  \end{equation}
  where $g_t$ is a one-parameter family of metrics
  on $V \setminus \tS$. Now let $u = 1 - \tau = e^{-t}$, so that
  $t = -\log u$. Then in these coordinates,
  \begin{equation}
    \label{pullbackmetric}
    (\exp|_{N_+(V \setminus \tS)})^*g = \frac{du^2 + u^2g_{-\log u}}{u^2}.
  \end{equation}
  Set $h_u = u^2g_{-\log u}$, so that this takes the form (\ref{normformeq}). 
  Now $u$ obviously extends to a global coordinate on $\cbundle{N_+(V \setminus \tS)}$, and we have
  already observed that $\exp$ is a diffeomorphism. Taking $\psi(u,q) = \exp(1-u,q)$ and $\tU = \psi(\cbundle{N_+(V \setminus \tS)})$,
  we plainly have $\psi(\left\{ u = 0 \right\}) = \tU \cap \tM$, $\psi(\left\{ u = 1 \right\}) = \tU \cap \tQ$, and
  $\psi|_{\{1\} \times V} = \id$. Uniqueness
  of $\psi$ follows from uniqueness of the form (\ref{geodesicform}), because all the intermediate steps are reversible.
  To prove Theorem \ref{normform}, it thus remains only to show that $h_u$ extends smoothly down to $u = 0$ 
  and that each $h_u$ is a conformally compact metric on $V$.

  Since $\psi$ is a diffeomorphism, we do our calculation on $\tU$ and regard $u$ as a function on $\tU$. For $0 \leq c \leq 1$ set $V_c =
  \left\{ x \in \tU: u(x) = c \right\} \approx V$. Let $\{x^s\}$ be a local coordinate system on $S$, and $(\theta,x^s,\rho)$ a polar identification. Then 
  locally, $u = u(\theta,x^s,\rho)$, so taking the exterior derivative of both sides of
  the equation $u = c$, we find that along $V_c$,
  \begin{equation*}
    0 = du = \frac{\partial u}{\partial \theta}d\theta + \frac{\partial u}{\partial x^s}dx^s + \frac{\partial u}{\partial \rho}d\rho.
  \end{equation*}
  Now $\frac{\partial u}{\partial \theta} \neq 0$, and so for some smooth functions $a, b_s$, we have
  \begin{equation}
    \label{dtheta}
    d\theta = ad\rho + b_sdx^s
  \end{equation}
  on $V_c$. Notice that $\frac{u}{\sin\theta}$ is smooth and nonvanishing on $\tU$. Then by (\ref{polform}),
  \begin{align*}
    u^2g &= \frac{u^2}{\sin^2(\theta)}\left( d\theta^2 + \frac{d\rho^2 + k_{\rho}}{\rho^2} \right) + (\rho u^2\sin(\theta)\ell)\\
    &=\frac{u^2}{\sin^2(\theta)}\frac{\rho^2d\theta^2 + d\rho^2 + k_{\rho}}{\rho^2} + (\rho u^2\sin(\theta)\ell).
  \end{align*}
  It is then clear by (\ref{dtheta}) that the restriction of $u^2g$ to any $V_c$ gives a smooth conformally compact metric
  on $V$ depending smoothly on $c$ all the way up to $c = 0$. Thus Theorem \ref{normform} is proved.

  Now suppose that $Q$ makes a constant angle $\theta_0$ with $M$, so that $\tQ \cap \tS$ is given by
  $\left\{ \theta = \theta_0 \right\}$. Let $\alpha = (\csc\theta_0 - \cot\theta_0)^{-1}$, and define
  a coordinate $\phi$ on $\cbundle{N_+(V\setminus\tS)}$ by
  $u = \alpha(\csc(\phi) - \cot(\phi))$. Notice that $\frac{du}{u} = \frac{d\phi}{\sin\phi}$.
  Thus, the metric (\ref{normformeq}) transforms to
  \begin{equation}
    \label{pullbackconst}
    \chi^*g = \frac{d\phi^2 + l_{\phi}}{\sin^2(\phi)},
  \end{equation}
  where we have defined $\chi:[0,\theta_0] \times V \to \tU$ by $\chi(\phi,q) = \psi(u(\phi),q)$ and
  $l_{\phi} = \frac{\sin^2\phi}{u^2}h_{u(\phi)}$.
  We view $\phi$ as a function on $\tU$ via the diffeomorphism $\chi$.
  Now by (\ref{ssoln}), we see that $\theta = \phi$ on $[0,1] \times (\tQ \cap \tS)$;
  or put differently, that $\theta = \phi + O(\rho)$. Because the level sets of $\phi$ and $u$ are the same,
  we get (\ref{dtheta}) again, and so by an identical calculation to the
  preceding, we find that
  \begin{equation*}
    \sin^2(\phi)g = \frac{\sin^2(\phi)}{\sin^2(\theta)}\frac{\rho^2d\theta^2 + d\rho^2 + k_{\rho}}{\rho^2} + (\rho\sin^2(\phi)\sin(\theta)\ell),
  \end{equation*}
  which is asymptotically hyperbolic as desired on each $V_c$ because $\frac{\sin\phi}{\sin\theta} = 1 + O(\rho)$. 

  We now prove the final claim.
  For this purpose, we define a change of coordinates $(\phi,y^{\mu})$ by
  setting $y^{\mu} = x^{\mu}$ on $\tQ$ and extending $y^{\mu}$ to be constant along orbits of the exponential map. These are coordinates by Theorem
  \ref{mainthm}. We wish to show that $(y^{n})^2l_{\phi}|_{\rho = 0}$ is constant in $\phi$.
  Now notice that it follows by (\ref{quadest}) that for $1 \leq \mu \leq n$,
  \begin{equation*}
    \frac{\partial}{\partial y^{\mu}} = \frac{\partial}{\partial x^{\mu}} + \frac{\partial \theta}{\partial y^{\mu}}\frac{\partial}{\partial \theta} + O(\rho).
  \end{equation*}
  Since $\frac{\partial}{\partial \theta} \in \ker \bar{g}$ at $\rho = 0$, we have, for $p \in \tU$, that
  \begin{equation*}
    \left.\bar{g}\left( \frac{\partial}{\partial y^{\mu}},\frac{\partial}{\partial y^{\nu}} \right)\right|_p = \left.\bar{g}\left( \frac{\partial}{\partial x^{\mu}},\frac{\partial}{\partial x^{\nu}} \right)\right|_p + 
    O(\rho(p)).
  \end{equation*}
  The right-hand side is constant in $\phi$ for $\rho = 0$ by (\ref{g}). But
  \begin{equation*}
    (y^n)^2l_{\phi}\left(\frac{\partial}{\partial y^{\mu}},\frac{\partial}{\partial y^{\nu}}  \right) =
    \bar{g}\left( \frac{\partial}{\partial y^{\mu}},\frac{\partial}{\partial y^{\nu}} \right)
  \end{equation*}
  at $\rho = 0$, and as we have seen, the right-hand side is constant in $\phi$ there. This yields the claim.

  Renaming $\phi$ by $\theta$, $\chi$ by $\psi$, and $l_{\phi}$ by $h_{\theta}$ yields the result.
\end{proof}

\begin{proof}[Proof of Corollaries \ref{normformcor} and \ref{normformconstcor}]
  We prove Corollary \ref{normformcor}. Let $V \subset \tQ$, $\tU \subset \tX$, and $\psi:[0,1] \times V \to \tU$ be as in Theorem \ref{normform}. Set
  $W = \psi(\left\{ 0 \right\} \times V)$, a neighborhood in $\tM$ of $\tM \cap \tS$.
  Let $\phi:V \to W$ be the diffeomorphism given by $\phi(q) = \psi(0,q)$. Define
  $\zeta(u,m) = \psi(u,\phi^{-1}(m))$.

  Uniqueness follows from uniqueness in Theorem \ref{normform}, the construction can be reversed to recover $\psi$ from $\zeta$.
\end{proof}

\begin{proof}[Proof of Corollary \ref{polarcor}]
Let $W, \tU, h_{\theta},$ and $\zeta$ be as in Corollary \ref{normformconstcor}.
Now $h_0$ is an asymptotically hyperbolic metric. So by the existence result for the standard geodesic
normal form (\cite{gl91}), there is a unique diffeomorphism $\varphi:S \times [0,\varepsilon)_{\rho} \to W$ such that
$\varphi^*h_0 = \rho^{-2}(d\rho^2 + k_{\rho})$, where $k_{\rho}$ is a smooth one-parameter family of metrics on $S$ such that
$k_0 = k$; and such that $\varphi|_{\left\{ 0 \right\} \times S} = id_S$.
The desired diffeomorphism is then given by $\chi = \zeta \circ (\id_{[0,\theta_0]} \times \varphi)$. The claimed properties follow immediately.
(Note that the $\rho$ in Corollary \ref{polarcor} is not the same as that appearing in polar $g$-coordinates).
\end{proof}

\bibliographystyle{alpha}
\bibliography{norm}
\end{document}